\documentclass[a4paper,12pt]{article}
\usepackage{fullpage}
\usepackage{graphicx}

\usepackage{titling}

\usepackage{amscd}
\usepackage{amsmath,amsthm, amssymb}
\usepackage{enumerate}
\usepackage{mathtools}
\usepackage[table]{xcolor}
\usepackage{color}
\usepackage[all]{xy}
\usepackage{ytableau}

\newcommand{\Gl}{\mathrm{GL}}
\DeclareMathOperator{\Hom}{\mathrm{Hom}}

\newcommand{\N}{\mathbb{N}}

\newcommand{\alg}{\mathrm{alg}}

\newcommand{\be}{\overline{e}}

\newcommand{\cP}{\mathcal{P}}
\newcommand{\cQ}{\mathcal{Q}}
\DeclareMathOperator{\ch}{\mathrm{char}}
\newcommand{\df}{\mbox{-}}

\newcommand{\K}{\mathbb{K}}
\newcommand{\ledom}{\trianglelefteq}

\newcommand{\mdl}[1]{M_{#1}}
\newcommand{\md}{\mbox{-mod}}
\newcommand{\module}[1]{\ensuremath{M_{#1}}}
\newcommand{\poset}{\mathcal{P}}
\newcommand{\quiver}{\cQ}
\DeclareMathOperator{\rad}{\mathrm{rad}}

\newcommand{\rep}{\df\mathbf{rep}}

\newcommand{\rst}[2]{#1|_{_{#2}}}
\DeclareMathOperator{\sep}{\mathrm{sp}}
\newcommand{\sm}[1]{\ensuremath{\left(\begin{smallmatrix} #1 \end{smallmatrix}\right)}}

\DeclareMathOperator{\supp}{\mathrm{supp}}
\newcommand{\s}[1]{\begin{smallmatrix} #1 \end{smallmatrix}}
\newcommand{\td}[1]{\underline{#1}}

\newcommand{\xx}[2]{\ensuremath{\xi_{ {#1,#2}}} }

\newtheorem{theorem}{Theorem}[section]

\newtheorem{proposition}[theorem]{Proposition}

\theoremstyle{definition}
\newtheorem{definition}[theorem]{Definition}

\theoremstyle{remark}

\numberwithin{equation}{section}

\newcommand{\si}[1]{\small{\emph{#1}}}

\preauthor{}
\postauthor{}
\DeclareRobustCommand{\authorthing}{
\begin{center}
	\begin{tabular}{lll}
		Karin Erdmann & Ana Paula Santana &
		Ivan Yudin\\	
\si{Mathematical}
& \si{CMUC, Department of} & \si{CMUC, Department of}\\
\si{Institute}&
\si{Mathematics}&
\si{Mathematics}\\
\si{University of Oxford}
&\si{University of Coimbra}
&\si{University of Coimbra}\\
\si{Oxford}
&\si{Coimbra} & \si{Coimbra}\\
\si{UK} & \si{Portugal} &\si{Portugal}\\
\si{erdmann@maths.ox.ac.uk} & \si{aps@mat.uc.pt} & \si{yudin@mat.uc.pt}
	\end{tabular}
\end{center}
}
\date{}
\author{\authorthing}

\title{Representation type of Borel-Schur algebras}

\begin{document}
\maketitle
\abstract{In our previous work~\cite{finite}, we found all Borel-Schur
algebra of finite representation type. In the present article, we determine which Borel-Schur
algebras of infinite representation type are tame, and which are wild. }
\section{Introduction}

Borel-Schur algebras occur as subalgebras of Schur algebras.
They were introduced by J. A. Green in \cite{green2}. A main result, or even the
motivation, of that work,  is that the Schur algebra has a triangular decomposition with factors an upper and a lower Borel-Schur algebra.

Borel-Schur algebras have shown to be a    powerful
tool for the study of projective resolutions of Weyl modules for the general linear
group  \cite{woodcock_vanishing, woodcock}.
More recently they
played a crucial role in the work of the last two authors~\cite{apsiy} on this problem. Also, in the same paper, Borel-Schur algebras were used to prove the Boltje-Hartmann  conjecture \cite{BH} on
permutational resolutions of (co-)Specht modules.

We fix an infinite field  $\K$.
The Schur algebra $S(n,r)$ is the  $\K$-algebra whose mo\-du\-le
category is equivalent to the category of $r$-homogeneous
polynomial representations of the  general linear group
$ \Gl_n(\K)$. This Schur algebra is finite-dimensional, and it
has an explicit subalgebra $S^+(n,r)$, the (upper) Borel-Schur algebra,   whose module category is equivalent
to the category of $r$-homogeneous polynomial re\-pre\-sen\-tations of
$B^+$, the group of upper triangular matrices in $\Gl_n(\K)$.
The para\-me\-ters of $S^+(n, r)$  which we
have to take into account are  $n$, $r$, and, in addition,
the characteristic of the field $\K$.

Borel-Schur algebras  are basic. They have finite global dimension and a highest weight theory.
There is an explicit formula for the multiplication, but except for small
cases, this is not easy to use.
In \cite{aps},  the second author determined the Ext quiver of a Borel-Schur
algebra. She also proved some results useful for the  construction of almost split sequences of simple
modules.
More recently this was continued in \cite{finite}, where it was also 
determined precisely
which Borel-Schur algebras are of finite representation type.
The answer is:

\begin{theorem}[{\cite{finite}}]\label{nossa}
Consider the Borel-Schur algebra $S^+(n,r)$ over an algebraically closed
field~$\K$. Then $S^+(n,r)$ has finite representation type if and only if
\begin{enumerate}[1)]
\item $n=2$ and one of the following alternatives holds:
\begin{enumerate}
\item $\ch(\K)=0$;
 \item$\ch(\K)=2$ and $r\leq 3$;
 \item $\ch(\K)=3$ and $r\leq 4$;
 \item $\ch(\K)=p\geq 5$ and $r\leq p$;
 \end{enumerate}
\item $n \ge 3$ and $r=1$. 
\end{enumerate}
\end{theorem}

This leaves to identify  when a Borel-Schur algebra has tame representation type, and this is answered completely in
this paper.  Our main result is  as follows.

\begin{theorem} Consider the Borel-Schur algebra $S^+(n,r)$ over an algebraically closed field~$\K$. Suppose that  $S^+(n,r)$ is of infinite type. Then
$S^+(n,r)$ is tame if  
\begin{enumerate}[(a)]
\item $n=2$, $\ch\K=3$, and $r=5$;
\item $n=3$ and $r=2$.
\end{enumerate}
Otherwise $S^+(n,r)$ is wild.
\label{main}
\end{theorem}

The above results can be visualized as follows:
\begin{equation*}
\begin{gathered}
\ch\K =0\\
\xygraph{
	!{<0cm,0cm>;<1.7ex,0cm>:<0cm,1.7ex>::}
:@{-}[rdrd]
:@{-}[uu]
:@{}|(0.3){r}|(0.7){n}[ldld]
:@{-}[rrrrrrrrrr]
[lllllllu]*{1}
[rr]*{2}
[rr]*{3}
[rr]*{\ldots}
[llllllld]:@{-}[dddddddd]
[uuuuuuul]*{2}
[dd]*{3}
[dd]*{4}
[dd]*{\ldots}
[uuuuuurrrr]*{finite}
[dl]:@{-}[rrrrrr]
[llllll]:@{-}[dddddd]
[uuuu]:@{-}[rr]:@{-}[uu]
[ld]*{t}
[rrrddd]*{wild}
}
\end{gathered}
\quad\quad\quad\quad\quad
\begin{gathered}
\ch\K =2\\
\xygraph{
	!{<0cm,0cm>;<1.7ex,0cm>:<0cm,1.7ex>::}
:@{-}[rdrd]
:@{-}[uu]
:@{}|(0.3){r}|(0.7){n}[ldld]
:@{-}[rrrrrrrrrrrr]
[lllllllllu]*{1}
[rr]*{2}
[rr]*{3}
[rr]*{4}
[rr]*{\ldots}
[llllllllld]:@{-}[dddddddd]
[uuuuuuul]*{2}
[dd]*{3}
[dd]*{4}
[dd]*{\ldots}
[uuuuuurrrr]*{finite}
[dl]:@{-}[rrrr]
[llll]:@{-}[dddddd]
[uuuu]:@{-}[rr]:@{-}[uu]
[ld]*{t}
[rrrddd]*{wild}
[uuuu]:@{-}[uu]
}
\end{gathered}
\end{equation*}
\bigskip
\begin{equation*}
\begin{gathered}
\ch\K =3\\
\xygraph{
	!{<0cm,0cm>;<1.7ex,0cm>:<0cm,1.7ex>::}
:@{-}[rdrd]
:@{-}[uu]
:@{}|(0.3){r}|(0.7){n}[ldld]
:@{-}[rrrrrrrrrrrrrrrr]
[lllllllllllllu]*{1}
[rr]*{2}
[rr]*{3}
[rr]*{4}
[rr]*{5}
[rr]*{6}
[rr]*{\ldots}
[llllllllllllld]:@{-}[dddddddd]
[uuuuuuul]*{2}
[dd]*{3}
[dd]*{4}
[dd]*{\ldots}
[uuuuuurrrr]*{finite}
[dl]:@{-}[rrrrrrrr]
[llllllll]:@{-}[dddddd]
[uuuu]:@{-}[rr]:@{-}[uu]
[ld]*{t}
[rrrddd]*{wild}
[uuuurr]:@{-}[uu]
[rd]*{t}
[rd]:@{-}[uu]
}
\end{gathered}
\quad\quad
\begin{gathered}
\ch\K =p\ge 5\\
\xygraph{
	!{<0cm,0cm>;<1.7ex,0cm>:<0cm,1.7ex>::}
:@{-}[rdrd]
:@{-}[uu]
:@{}|(0.3){r}|(0.7){n}[ldld]
:@{-}[rrrrrrrrrrrrrrrr]
[lllllllllllllu]*{1}
[rr]*{2}
[rr]*{\ldots}
[rr]*{p\phantom{1}}
[rrr]*{p+1}
[rrr]*{\ldots}
[llllllllllllld]:@{-}[dddddddd]
[uuuuuuul]*{2}
[dd]*{3}
[dd]*{4}
[dd]*{\ldots}
[uuuuuurrrr]*{finite}
[dl]:@{-}[rrrrrr]
[llllll]:@{-}[dddddd]
[uuuu]:@{-}[rr]:@{-}[uu]
[ld]*{t}
[rrrddd]*{wild}
[uuuurr]:@{-}[uu]
}
\end{gathered}
\end{equation*}
\bigskip

\noindent To prove Theorem~\ref{main}, we reduce the problem by idempotent methods: one has
to show that only few algebras are wild, and that the algebras listed in the theorem are tame.
To prove that an algebra is wild can be   done by
relating its module category to that of some known wild algebra.
Our main method is based on coverings (see Section 2).

This leaves to prove that the two remaining algebras are tame.
We show  that $S^+(3, 2)$ degenerates to a special biserial algebra, which is known to be tame. Then  a result from  \cite{geiss} implies that the algebra
$S^+(3,2)$ is tame.  Our proof works for arbitrary characteristic, although the algebra structure for characteristic $2$ is
different.

Our proof  that  $S^+(2, 5)$ in characteristic $3$ is tame is very different from the proof of the previous case. This is done by exploiting  representation theory of posets, and using the fact that
the representation type of posets is completely understood.
To follow this route, it is crucial   that  $S^+(2, r)$ is a one-point extension of $S^+(2, r-1)$ (see Section 6 for details).

In general, given  a one-point extension $A[M]=\left(\begin{matrix} A & M\cr 0&\K\end{matrix}\right)$, if
$A$ has finite type and $M$ is suitable, one can construct a finite poset from the
Auslander-Reiten quiver of $A$. Moreover, the representation type of this poset is the same as
the representation type of the algebra $A[M]$.
We use this when $A= S^+(2, 4)$, $A[M]=S^+(2, 5)$, and the characteristic of
the base field is $3$.
It was proved in~\cite{finite}  that the algebra $ S^+(2, 4)$  has finite type  by explicitly computing its Auslander-Reiten quiver. Now we take this quiver, compute the
relevant poset, and then prove that it has tame type (see Section 6).
\medskip

We note that the representation type of Schur algebras $S(n,r)$ (and also of their q-analogs) has been classified,
but the methods used  are different (see \cite{DEMN} and \cite{EN}), and also as
far as we can see there is no connection with the techniques we use in the
present article.
This makes the problem we describe next rather intriguing. Besides Borel-Schur
algebras, Green defined in~\cite{green2} a subalgebra $S(G,r)$ of
$S(n,r)$ for every subgroup $G$ of $\Gl_n(\K)$. This subalgebra coincides
with $S^+(n,r)$ if $G=B^+$  and with $S(n,r)$ if $G=\Gl_n(\K)$. 
Of course, $B^+$ and $\Gl_n(\K)$ are extremal elements of the family of
parabolic subgroups $P_\lambda$ in $\Gl_n(\K)$.  
It would be interesting to determine the representation type of the algebras
$S(P_\lambda,r)$ with arbitrary $P_\lambda$.
\medskip

 The paper is organized as follows. In Section~2 we introduce the techniques we will use on the study of wild type. 
 Section~3 is  dedicated to Borel-Schur algebras. We introduce  basic facts  and prove  that for every positive integer $m\le n$ and $s\le r$ there is an idempotent $e$ in $S^+(n,r)$
such that
$\,e S^+(n,r) e \cong S^+(m,s)\,$. This  result will be crucial for our
classification. In Section~4 we classify the Borel-Schur algebras of wild
representation type using the covering techniques described in  Section~2.
Using degeneration techniques due to Gabriel~\cite{gabrieldegen} and  Geiss~\cite{geiss}, we prove in Section~5 that $S^+(3,2)$  is tame. Section~6 contains the proof that   $S^+(2, 5)$  is tame over fields of characteristic $3$.
As was mentioned above this is done using representation theory of posets.

For background on Schur algebras and Borel-Schur algebras we refer to~\cite{green} and~\cite{green2}.
Background on representation theory of algebras can be found in \cite{ringel1099}, or other text books.

We assume throughout that the field $\K$ is algebraically closed, and that all quivers are finite.

\section{Preliminaries on wild representation type}
Let $A$ be a finite dimensional algebra over $\K$. We will write $A\md$ for the category of finite dimensional left $A$-modules.
The algebra $A$ is said to have \emph{finite representation type} if there are only finitely many
isomorphism classes of finite dimensional indecomposable modules. Otherwise $A$ has   \emph{infinite representation type}.  The famous Drozd Dichotomy
Theorem, proved in~\cite{drozd}, divides algebras of  infinite representation type  into two mutually exclusive classes: algebras of  \emph{tame type } and algebras of \emph{wild type}.
The algebra  $A$ is tame  if it has infinite type and, for every dimension $d \geq 0$, all, but a finite number of, isomorphism classes of indecomposable $A$-modules of dimension $d$ can be parametrised  by a finite number of 1-parameter families.

To define wild we need a further  notion.

Given another $\K$-algebra $B$,
a functor $F\colon B\md \to A\md$ is called a \emph{ representation embedding} if it
preserves indecomposability and isomorphism classes. More formally, $F$ is a
representation embedding if for every indecomposable object $X\in B\md $ the
object $F(X)$ is indecomposable in $A\md $, and if $F(Y) \cong F(Z)$ for some
$Y$ and $Z$ in $B\md $, then $Y\cong Z$.

An algebra $A$ is \emph{wild} if there is an $A$-$\K\left\langle u,v
\right\rangle$-bimodule $Z$, free of finite rank as a right $\K\left\langle u,v
\right\rangle$-module, such that the functor $Z \otimes_{\K\left\langle u,v
\right\rangle}- \colon \K\left\langle u,v \right\rangle\md \to A\md$ is a
representation embedding.

It is shown in~\cite[Proposition~22.4]{bautista} that  to prove that $A$ is wild it is enough to see that the  functor $Z\otimes_{\K\left\langle u,v
\right\rangle}-$ preserves isomorphism classes.
As an  immediate corollary we get that if $B$ is a wild algebra and there is an
isomorphisms preserving functor $B\md\to A\md$, then $A$ has wild representation
type.

Since full and faithful  functors preserve isomorphism classes the following result is
obvious.
\begin{proposition}
\label{quotient}
Let $A$ be a
finite dimensional algebra over $\K$. Suppose there is an ideal  $I$ of  $A$ such
that $A/I$ has wild representation type. Then $A$ is a wild algebra.
\end{proposition}
It is not in general true that if $B$ is a wild subalgebra of an algebra
$A$, then $A$ is also wild. Nevertheless,  the following partial result
in this direction holds.
\begin{proposition}
\label{subalgebra}
Let $A$ be a finite dimensional algebra. Suppose there is an idempotent $e$  in $A$ such that
$eAe$ is wild. Then $A$ has  wild representation type.
\end{proposition}
\begin{proof}
The induction functor $M \mapsto Ae\otimes_{eAe} M$ from
$eAe\md$ to $A\md$ preserves isomorphism classes. In fact,
suppose $Ae \otimes_{eAe} M \cong Ae \otimes_{eAe} N$ for some $M$, $N\in
eAe\md$. Then
\begin{equation*}
M \cong eAe\otimes_{eAe} M = e (Ae \otimes_{eAe} M) \cong e (Ae
\otimes_{eAe} N) \cong N.\qedhere
\end{equation*}
\end{proof}
We will now describe further sufficiency criteria for wildness of basic
algebras in terms of their quivers. For this we need some notation.

For a (finite) quiver $\quiver$ and a $\K$-algebra $B$, we denote by
$\quiver\rep_B$ the category of representations $((V_x)_{x\in \quiver_0},
(V_\alpha)_{\alpha \in \quiver_1})$ such that $V_x$ is a finitely generated free
$B$-module for every $x \in \quiver_0$. Given a collection $P$ of paths in $\quiver$ with common source
and target, we say that $\sum_{p \in P} b_p p$, with $b_p\in B$, is a
\emph{relation defined over $B$}. For a collection $R$ of relations  defined
over $B$, we write $(\quiver,R)\rep_B$ for the full subcategory of
$\quiver\rep_B$ whose objects are the representations on which every relation in
$R$ vanishes. 
We suppress $B$ when it coincides with the base field $\K$.
The objects of $(\quiver,R)\rep$ are of course just  \emph{finite dimensional
representations of $(\quiver,R)$}.
 Evidently $(\quiver,R)\rep$ is equivalent to the category of
finite dimensional modules over $\K\quiver/ \left\langle R \right\rangle$, where
$\K\quiver$ is the path algebra of $\quiver$ and $\left\langle R
\right\rangle$ the ideal generated by $R$ in this algebra.

The algebra $\K\quiver/ \left\langle R \right\rangle$ is wild if and only if 
 there exists  $Z \in (\quiver,R)\rep_{\K\left\langle u,v
\right\rangle}$ such that the
 functor \begin{equation*}Z \otimes_{\K\left\langle u,v
\right\rangle} - \colon
{\K\left\langle u,v
\right\rangle}\md \to (\quiver,R)\rep\end{equation*}
defined by
\begin{equation*}
\left( Z\otimes_{\K\left\langle u,v
\right\rangle} V \right)_x = Z_x \otimes_{\K\left\langle u,v
\right\rangle} V_,
\quad
\left( Z\otimes_{\K\left\langle u,v
\right\rangle} V \right)_a = Z_a \otimes_{\K\left\langle u,v
\right\rangle} V,\,\,\,\mbox{all} \,\,\,x\in \quiver_0, \,\alpha \in \quiver_1,
\end{equation*}
is a representation
embedding.
We say 
 that $(\quiver, R)$ is wild if  the corresponding path algebra  is wild.

Let $(\quiver, R)$ be a quiver with relations defined over $B$ and $\quiver'$ a subquiver of
$\quiver$. Suppose $r = \sum_p a_p p$ is a relation in $\quiver$. We define the restriction $\rst{r}{\quiver'}$ of $r$ to $\quiver'$ by
$\rst{r}{\quiver'}=\sum\limits_{p
\mbox{ in } \quiver'} a_p p$. In particular, if the initial or the final vertex of
$r$ is not in $\quiver'$ then $\rst{r}{\quiver'}=0$.
 Denote by $\rst{R}{\quiver'}$ the collection $\left\{\,
\rst{r}{\quiver'}
\,:\, r\in R,\ \rst{r}{\quiver'}\not=0
\right\}$ of relations in $\quiver'$. If $V$ is a representation of
$(\quiver',\rst{R}{\quiver'})$ over $B$ then, following~\cite{karin_book}, we define the \emph{extension-by-zero} representation
$\widetilde{V}$ of $(\quiver, R)$ over $B$ by
\begin{equation*}
\begin{aligned}
\widetilde{V}_x & =
\begin{cases}
V_x, & x\in \quiver'\\
0, & \mbox{otherwise}
\end{cases} &\quad \quad
\widetilde{V}_a & =
\begin{cases}
V_\alpha, & \alpha \in \quiver'\\
0, & \mbox{otherwise}
\end{cases}
\end{aligned}
\end{equation*}
for every vertex $x$ and every arrow $\alpha$ in $\quiver$.
It is not difficult to see that  the correspondence $V \mapsto
\widetilde{V}$ defines a full and faithful functor  from $(\quiver',
\rst{R}{\quiver'})\rep_B$ to
$(\quiver,R)\rep_B$. This implies the following result.
\begin{theorem}\label{subquiver}
Let $\quiver$ be a quiver and $R$ a set of relations defined over $\K$ in $\quiver$. Suppose
$\quiver'$ is a subquiver of $\quiver$ such that $(\quiver',\rst{R}{\quiver'})$ is wild. Then
$(\quiver,R)$ is wild.
\end{theorem}
Given $ V\in
\quiver\rep_B$, we define $\supp(V)$  to be the subquiver of $\quiver$ containing all the
vertices $x\in \quiver_0$ and all the arrows $\alpha \in \quiver_1$ such that
$ V_x\not=0$ and $V_\alpha \not=0$.
Then the essential image of the extension-by-zero functor from
$(\quiver',\rst{R}{\quiver'})\rep_B$  to $(\quiver,R)\rep_B$
coincides with the class of those representations $V$ such that $\supp(V) \subset
\quiver'$.
We will use this fact in the proof of Theorem~\ref{covering_criterion}.

Next we discuss the behaviour of representation type under coverings of quivers.
Let $\cQ$ be a quiver equipped with an action of a finite group $G$.
Denote by $\phi$ the canonical projection
$ \cQ \to
\cQ/G $.
In that situation one says that $\cQ$ is  a \emph{regular covering} of
$\cQ/G$.
Given a representation $V$ of $\cQ$ we define the representation $\phi_* V$  of $\cQ/G$  by
\begin{equation*}
\phi_* (V)_{xG} := \bigoplus_{g \in G} V_{xg},\quad\quad
\phi_* (V)_{\alpha G} = \sum_{g \in G} \varepsilon_{yg} \circ V_{\alpha g} \circ \pi_{x g},
\end{equation*}
for every $xG \in \cQ/G$ and
every arrow $x \xrightarrow{\alpha} y$ in $\cQ$,
 where $\pi_{z}$ and
$\varepsilon_{z}$ are, respectively,  the canonical projections and inclusions associated with
the direct sum decomposition.

The group $G$ induces an action on the category of finite dimensional representations of $\cQ$ as follows. Given such a representation
$V$ and $g\in G$ we define
\begin{equation*}
 (g_* V)_x = V_{xg^{-1}},\quad
(g_* V)_{\alpha} = V_{\alpha g^{-1}} \colon (g_* V)_x= V_{xg^{-1}} \to
(g_* V)_y = V_{yg^{-1}}.
\end{equation*}
For the convenience of the reader we restate \cite[Lemma~3.5]{gabrielLNM}
\begin{theorem}[{\cite{gabrielLNM}}]\label{gabriel}
Let $\cQ$ be a quiver and $G$ a group acting freely on $\cQ$.
Suppose that
$V$ is a finite dimensional indecomposable representation of $\cQ$ over
$\K$ such that
$g_* V \not\cong V$, for every $g\in G$,  $g\neq 1_G$.
Then $\phi_* V$ is indecomposable. Moreover, if $U\not\cong V$ is a representation of
$\cQ$ such that $\phi_* U \cong \phi_* V$, then there is $g\in G$, $g\not=1_G$, such
that $g_* V \cong U$.
\end{theorem}
Let  $G$ be a finite group acting freely  on $\quiver$.
Then for every vertices $x$, $y\in \quiver$, path $p\colon xG \to yG$ in
$\quiver/G$, and  $g\in G$ there
is a unique element $g'\in G$ such that there is a (unique) path $p_g \colon xg \to yg'$ in $\quiver$ that lifts
$p$. Notice that the correspondence $g\mapsto g'$ defines a bijective map
$\sigma_p \colon G \to G$.
For each relation $r=\sum\limits_{p \in \cal{J}} a_p p$ in $\quiver/G$ we define the set of relations
$r_G$ in $\quiver$ by
\begin{equation*}
r_G := \Big\{\, \sum_{ p \in \cal{J}} a_p p_g \,\Big|\, g\in G \Big\}.
\end{equation*}
It is not difficult to see that if $V$ is a representation of $(\quiver,r_G)$
then $\phi_* V$ is a representation of $(\quiver/G, r)$. More generally, if
$R$ is a set of relations in $\quiver/G$ and we define $R_G = \bigcup\limits_{r\in R} r_G$, then
for every representation $V$ of $(\quiver, R_G)$, the representation $\phi_* V$
of $\quiver/G$ satisfies the relations in $R$.

The next theorem will be used to prove wildness of some quivers with relations.
We say that a wild quiver with relations $(\quiver,S)$ is \emph{minimal wild} if $(\quiver',
\rst{S}{\quiver'})$ is not wild for every proper subquiver $\quiver'$ of
$\quiver$. 

\begin{theorem}\label{covering_criterion}
Let $\quiver$ be a finite quiver, $G$ a group acting freely on $\quiver$, and $S$ a set
of relations in $\quiver/G$.
Suppose there is a  subquiver $\quiver'$ of $\quiver$ such that
\begin{enumerate}[$\bullet$ ]
\item $(\quiver', \rst{S_G}{\quiver'})$ is minimal wild;
\item there is no non-trivial $g \in G$ that fixes $\quiver'$;
\end{enumerate}
Then the quiver $(\quiver/G, S)$ has wild representation type.
\end{theorem}
\begin{proof}
Let $Z$ be a representation of $(\quiver', \rst{S_G}{\quiver'})$ over $\K\left\langle u,v
\right\rangle$ such that the functor $Z\otimes_{\K\left\langle u,v
\right\rangle} - $ is a representation embedding.
Recall that we denote by $\,\widetilde{\cdot }\,$ the extension-by-zero functor.
It is clear that the functors $\phi_* \circ
(\widetilde{Z}\otimes_{\K\left\langle u,v \right\rangle} -)$ and $(\phi_*
\widetilde{Z})\otimes_{\K\left\langle u,v \right\rangle} -$ are naturally isomorphic. Thus
to show that $(\quiver/G,S)$ is wild it is enough to check that $\phi_* \circ
(\widetilde{Z}\otimes_{\K\left\langle u,v \right\rangle}-)$ is a representation embedding.

Since $(\quiver', \rst{S_G}{\quiver'})$ is minimal wild, the support of
$Z$, and thus also the support of $\widetilde{Z}$,  coincides with $\quiver'$. 

Let $X$ be an indecomposable $\K\left\langle u,v \right\rangle$-module. Denote
$\widetilde{Z}\otimes_{\K\left\langle u,v \right\rangle}X $ by $\overline{X}$. Then
$\supp \left( \overline{ X }\right) = \supp
\left( \widetilde{Z} \right) = \quiver'$ and  so $\supp(g_* \overline{X}) =
\supp(\overline{X}) g^{-1}
\not= \supp(\overline{X}) $ unless $g=e_G$. In particular, $g_*
\overline{X} \not\cong \overline{X}$.
Thus, by Theorem~\ref{gabriel}, the representation $\phi_* \overline{X}$ of
$(\quiver/G, S)$ is indecomposable.

Now let $Y$ be another indecomposable $\K\left\langle u,v \right\rangle$-module not isomorphic to $X$.
We write $\overline{Y}$ for
$\widetilde{Z}\otimes_{\K\left\langle u,v \right\rangle}Y$. Then $\overline{X} \not\cong
\overline{Y}$ since $\widetilde{Z}\otimes_{\K\left\langle u,v \right\rangle} - $ is a
representation embedding. Therefore, by Theorem~\ref{gabriel}, the
representations $\phi_* \overline{X}$ and $\phi_* \overline{Y}$ of
$(\quiver/G, S)$ are isomorphic if and only if there is $g\not=e_G$ such that
$g_* \overline{X} \cong  \overline{Y}$. But then
\begin{equation*}
\supp(\widetilde{Z}) = \supp(\overline{Y}) = \supp(g_* \overline{X}) =
\supp(\overline{X}) g^{-1} = \supp(\widetilde{Z})g^{-1},
\end{equation*}
which contradicts  the already proved assertion that $\supp(\widetilde{Z})g \not= \supp(\widetilde{Z}) $
for $g\not=e_G$. This finishes the proof of the theorem.
\end{proof}

Let $(\quiver,R)$ be a quiver with relations and
$A$ the corresponding basic algebra. Suppose $\quiver'$ is a subquiver of
$\quiver$. Denote by $e$ the idempotent of $A$ given by $\sum_{x\in \quiver'}
e_x$. Then $eAe$ is a basic algebra. It is not true in general  that the quiver of $e A e$ is $\quiver'$. We
say that $\quiver'$ is \emph{convex} if every path in $\quiver$ connecting two
vertices in $\quiver'$ completely lies in $\quiver'$. Notice, that a convex
subquiver is always full.
\begin{proposition}
\label{convex}
Let $(\quiver,R)$ be a quiver with relations defined over $\K$, and
 $\quiver'$  a convex subquiver of $\quiver$. Then $\quiver'$ is the
quiver of the basic algebra $e A e$, where $A = \K \quiver/ \left\langle R
\right\rangle$ and $e = \sum_{x\in \quiver'} e_x$.
\end{proposition}
\begin{proof}
The algebra $eAe$ is $\K$-spanned by the paths in $\quiver$
that start and end in $\quiver'$. As $\quiver'$ is convex all these paths lie
 inside $\quiver'$. Thus $eAe$ is generated as a ring by arrows in
$\quiver'$. Every arrow of $\quiver'$ lies in the radical of $A$ and so  is
nilpotent. Hence it also lies  in the radical of $eAe$.
Similarly every path of length no less than two with start and end in
$\quiver'$ lies in the $\rad^2 (eAe)$. This shows that
$\rad(eAe)/\rad^2(eAe)$ is generated by the arrows in $\quiver'$, i.e. that
$\quiver'$ is the quiver of the basic algebra $eAe$.
\end{proof}

\section{Borel-Schur algebras}
\label{schur_preliminaries}
In this section we introduce Schur and Borel-Schur algebras and establish some basic facts about these algebras.

Let  $n$ and $r$ be arbitrary  fixed positive integers. Consider  the general linear group $\Gl_n (\K)$ and denote by $\mbox{B}^+$ the Borel subgroup of $\Gl_n (\K)$ consisting of all  upper triangular invertible matrices.
The general linear group $\Gl_n (\K)$ acts on $\K^n$ by multiplication. So $\Gl_n (\K)$ acts on  the $r$-fold tensor product $(\K^n)^{\otimes r}$ by the rule $g (v_1 \otimes \dots \otimes v_r)= gv_1 \otimes \dots \otimes gv_r $, all $g \in \Gl_n (\K)$, $v_1, \dots, v_r \in \K^n$. Let
\begin{equation*}
\begin{aligned}
 \rho_{n,r} \colon \K \Gl_n(\K) \to  \mbox{End}_\K ((\K^n)^{\otimes r})
\end{aligned}
\end{equation*}
be the representation afforded by $(\K^n)^{\otimes r}$  as a $\K \Gl_n (\K)$-module. Then $\rho_{n,r}( \K \Gl_n(\K))$  is a subalgebra of $\mbox{End}_\K ((\K^n)^{\otimes r})$.
\begin{definition}  The algebra $\rho_{n,r}( \K \Gl_n(\K))$ is called the  Schur  algebra for $n$, $r$ and $\K$ and is denoted by $S_\K(n,r)$, or simply $S(n,r)$.
The Borel-Schur algebra $S^+(n,r)=S_\K^+(n,r)$ is the subalgebra $\rho_{n,r}( \K \mbox{B}^+)$ of the Schur algebra.
\end{definition}

To describe a standard basis for $S(n,r)$ we  need some combinatorics. We start by summarizing the terms we use.
\begin{itemize}
\item  $\Sigma_r$ is the symmetric group on $ \left\{\, 1, \dots, r \,\right\}$.
\item  $I(n,r)= \left\{\, i=(i_1,\dots,i_r) \,\middle|\,\, i_s \in \mathbb{Z}, \,\,1\leq i_s\leq n, \,\, \mbox{for all } \,s\,
\right\}$. The elements of $I(n,r)$ are called multi-indices.
\item $\Lambda(n,r)= \left\{\, \lambda=(\lambda_1,\dots,\lambda_n) \,\middle|\,\, \lambda_t \in \mathbb{Z}, \,\,0\leq \lambda_t \,\,(t=1, \dots, n), \,\,\sum_{t=1}^n \lambda_t=r
\right\}$ is the set of all compositions of $r$ into $n$ parts.
\item $i \in I(n,r)$ has \emph{weight} $\lambda \in \Lambda(n,r)$  if $ \lambda_t = \#\left\{\, 1\le s \le r \,\middle|\, i_s=t \right\}$,  $t=1, \dots, n$.
\item  $\ledom$ denotes the \emph{dominance order} on $\Lambda(n,r)$, that is
$\alpha\ledom\beta$ if
$\sum_{t=1}^s\alpha_t \le \sum_{t=1}^s \beta_t$, for all $1 \leq s \leq n$.
\item For $i,j \in I(n,r)$, $ i \leq j$ means $i_s \leq j_s$, $s=1,\dots, r$, and $i< j $ means $ i \leq j$ and $i\neq j$.
Obviously, if $ i $ has  weight  $\mu $   and $j$ has  weight  $\lambda$, then
$i \leq j $ implies $\lambda\ledom \mu$.
\end{itemize}

 The symmetric group $\Sigma_r$  acts on the right of $I(n,r)$ and of $I(n,r) \times I(n,r)$, respectively,   by
\begin{equation*}
i \sigma = (i_{\sigma(1)},\dots,i_{\sigma(r)})\,\,\, \mbox{and} \,\,\, (i,j)\sigma =(i\sigma, j\sigma),  \,\,\,\mbox{all}\,\, i,j \in I(n,r),  \sigma \in \Sigma_r.
\end{equation*}
Note that $i,j \in I(n,r) $ are in the same  $\Sigma_r$-orbit if and only if they have the same weight. Therefore
the $\Sigma_r$-orbits on $I(n,r)$ are identified  with the elements of~$\Lambda(n,r)$. We denote by $\Sigma_i$  the stabilizer of $i $ in $\Sigma_r$, that is $\Sigma_i= \left\{\, \sigma \in \Sigma_r \,\middle|\,\,i \sigma =i \,\right\}$. We write
$\Sigma_{i,j} = \Sigma_{i} \cap \Sigma_{j}$, all $i,j \in I(n,r)$.

To each pair $(i,j) \in I(n,r) \times I(n,r)$ one can associate  an element
$\xi_{i,j}$ of $S(n,r)$ (see~\cite{green}).  These elements satisfy $\xi_{i,j} =
\xi_{k,\ell}$ if and only if $( i,j )$  and $( k,\ell )$ are in the same
$\Sigma_r$-orbit of $I(n,r) \times I(n,r)$. Fix a
transversal  
$\Omega(n,r)$ for the action of $\Sigma_r$ on $I(n,r) \times I(n,r)$. Then   the set  $\left\{\,
\xi_{i,j}
\,\middle|\, (i,j)\in \Omega(n,r) \right\}$ is a basis of  $S(n,r)$ over $\K$.
It is also well known  (see~\cite{green2}) that $\,S^+(n,r) = \K\left\{\,
\xi_{i,j} \,\middle|\, i\le j,\ (i,j) \in \Omega(n,r) \right\}$.

A formula for the product of two basis elements is the following (see~\cite{green2}):
$\xi_{i,j}\xi_{k,h} = 0$, unless $j$ and $k$ are in the same $\Sigma_r$-orbit, and
\begin{equation}
	\label{0.0}
	\xi_{i,j}\xi_{j,h} = \sum_{\sigma} \left[ \Sigma_{i\sigma,
	h}:\Sigma_{i\sigma, j, h} \right] \xi_{i\sigma, h}
\end{equation}
where the sum is over a transversal $\left\{ \sigma \right\}$ of the set of all double cosets
$\Sigma_{i,j}\sigma \Sigma_{j,h}$ in $\Sigma_j$, and  $\Sigma_{i \sigma,j, h} = \Sigma_{i\sigma,h} \cap \Sigma_{j}$.

If $i$ has weight $\lambda\in \Lambda(n,r)$, we write $\xi_{i,i}= \xi_{\lambda}$. Then
$1_{S(n,r)} = \sum_{\lambda\in\Lambda(n,r)} \xi_{\lambda}$ is an orthogonal idempotent decomposition of $1_{S(n,r)} $.

It was shown in~\cite{aps} that the algebra $S^+(n,r)$ is a basic algebra.
The idempotents
$\xi_{\lambda}$ , $\lambda \in \Lambda(n,r)$, are primitive in
$S^+(n,r)$. The quiver $\quiver$ of $S^+(n,r)$  was \emph{de facto} determined
in~\cite[Theorem~5.4]{aps}.
 The vertices of $\cQ$ correspond to the primitive idempotents $\xi_\lambda$, and so $\quiver_0$ can be identified with $\Lambda(n,r)$.  If  $\ch \K=0$, the
quiver $\quiver$ contains  an arrow from the vertex $\lambda$ to  the vertex $\mu$ if and only if, for some positive integer $s$, we have
 $\mu -\lambda = \gamma_s$, where $\gamma_s =
(0,\dots, 1,-1,\dots, 0)$ with $1$ at the $s$th position. If  $\ch \K=p$, such an arrow exists  if and only if
there are integers $s \geq 1$ and $d \geq 0$ such that $\mu -\lambda = p^d \gamma_s$.
Notice that for such $\lambda$ and $\mu$ the vector space $\xi_\mu S^+(n,r)
\xi_{\lambda}$ is one-dimensional and is spanned by the element
$\xi_{\ell (\lambda(s, p^d)),\ell(\lambda)}$, where $ \ell(\lambda)$ and $\ell (\lambda(s, p^d)) \in I(n,r)$ are the standard elements
\begin{equation*}
\begin{aligned}
 \ell(\lambda) &= ( \underbrace{1,\dots,1}_{\lambda_1},
 \dots, \underbrace{n,\dots, n}_{\lambda_n}
) ,\\[2ex]
\ell (\lambda(s, p^d)) & =  ( \underbrace{1,\dots,1}_{\lambda_1},
 \dots, \underbrace{s, \dots, s}_{\lambda_s+p^d}, \underbrace{s+1, \dots, s+1}_{\lambda_{s+1}-p^d}, \dots,  \underbrace{n,\dots, n}_{\lambda_n}
).
\end{aligned}
\end{equation*}

A similar result holds in characteristic $0$, with $p^d$ replaced by $1$.
It should also be mentioned that the sets
\begin{equation}
\begin{aligned}
& \left\{\, \xi_{\ell (\lambda(s, 1)),\ell(\lambda)}\,\, \,\middle|\,\,1\leq s\leq
n-1\, \phantom{ 1 \leq p^d \leq \lambda_{s+1}} \,\,
\,\right\}, \,\, \mbox{if  $\ch\K=0$} ,\\[0.5ex]
\\
& \left\{\, \xi_{\ell (\lambda(s, p^d)),\ell(\lambda)} \,\middle|\,\,1\leq s\leq n-1;\,1 \leq p^d \leq \lambda_{s+1}\,
\right\}, \,\,\mbox{if $\ch\K=p$},
\end{aligned}
\end{equation}
are minimal sets of $S^+(n,r)$-generators of rad~$S^+(n,r)\xi_{\lambda}$ (see~\cite[Theorem~4.5]{aps}).
In~\cite{finite} we determined the relations for the quiver $\quiver$ of $S^+(2,r)$.
If $\ch \K =0$, $\quiver$ is of type $A_{r+1}$ and $\K\quiver \simeq S^+(2,r)$.
Suppose now that $\ch \K =p$.
For every $\lambda$, $\mu\in \Lambda(2,r)$ such that $\mu -\lambda =
(p^d,-p^d)$, we denote by $\alpha_{d,\lambda}$ the arrow from $\lambda$ to
$\mu$ in  $\quiver$. We say that
$\alpha_{d,\lambda}$ is of type $\alpha_d$. Notice that every vertex in
$\quiver$ has at most one incoming and at most one outgoing arrow of type
$\alpha_d$. This implies that to specify a path in $\quiver$ it is enough to
indicate the starting vertex and the types of arrows in the path. For example
$(\alpha_0 \alpha_1)_{(a,b)}$ will denote the path $\alpha_{0,(a+p,b-p)}
\alpha_{1,(a,b)}$. It is also natural to abbreviate the repeated types with the
usual power notation. With these conventions, the relations for $S^+(2,r)$ as
quotient algebra of the path algebra of $\quiver$ can be written as
\begin{equation*}
\begin{aligned}
(\alpha_s \alpha_t)_{\lambda} &= (\alpha_t \alpha_s)_{\lambda},\ \lambda_2 \ge
p^s+p^t, \,\,s\neq t\\[2ex]
(\alpha_s^{p})_\lambda = &0,\ \lambda_2\ge p^{s+1}.
\end{aligned}
\end{equation*}
\begin{proposition}
\label{idempotent}
For every $m\le n$ and $s\le t$ there is an idempotent $e$ in $S^+(n,t)$
such that
\begin{equation*}
e S^+(n,t) e \cong S^+(m,s).
\end{equation*}
\end{proposition}
\begin{proof}
It is enough to consider the cases $n=m+1$, $t=s$ and $n=m$, $t=s+1$. The case
$n=m+1$, $t=s$ was treated in Section~5 of~\cite{finite}.
For $n=m$ and $t=s+1$, we define
\begin{equation*}
e = \sum_{\lambda\in \Lambda(n, s+1)\colon \lambda_1 \ge 1}  \xi_\lambda.
\end{equation*}
For each  $i= (i_1, \dots, i_s) \in I(n,s)$, write $\bar \imath= (1, i_1, \dots,
i_s) \in I(n,s+1)$. Note that, for $i,j,k,\ell \in I(n,s)$, we have that $(i,j)$
and $(k, \ell)$ are in the same $\Sigma_s$-orbit of $I(n,s)\times I(n,s)$ if and
only if $(\bar \imath,\bar \jmath)$ and
$(\bar k,\bar \ell)$ are in the same $\Sigma_{s+1}$-orbit of $I(n,{s+1})\times
I(n,s+1)$. So there is an injective linear map $\phi \colon S^+(n,s) \to
S^+(n,{s+1})$   defined by
$
\phi (\xi_{ij}) = \xi_{\bar \imath, \bar \jmath}
$.
It is easy to see that   the image of $\phi$ coincides with $eS^+(n,{s+1})e$, and
that $\phi(1) = \phi(\sum_{\lambda \in \Lambda(n,s)}\xi_{\lambda}) = e$.
Next we verify that $\phi$ preserves products.

If $j$, $k \in I(n,s)$ are not in the same $\Sigma_s$-orbit, then $\bar \jmath$ and $\bar k$ are not in the same $\Sigma_{s+1}$-orbit.
Therefore $\, \xi_{i,j}\xi_{k,\ell} =0\,$ and $\, \xi_{\bar\imath,\bar \jmath}\xi_{\bar
k,\bar \ell} =0\,$. Now fix multi-indices $i \leq j \leq \ell$ in $I(n,s)$, and consider the products $\, \xi_{i,j}\xi_{j,\ell}\,$ and $\, \xi_{\bar\imath,\bar\jmath}\xi_{\bar\jmath,\bar \ell}\,$.
For each multi-index  $h$, let $\,h(t)= \left\{u  \,\middle|\,\,h_u=t\,
\right\}$, $t=1, \dots ,n$. Then the stabilizer of $h$ is $ \Sigma_{h(1)} \times \dots \times \Sigma_{h(n)} $. Since we know that $i \leq j \leq \ell$, we have $\ell(1) \subseteq j(1)\subseteq i(1)$, and $\Sigma_{i,j}= \Sigma_{j(1)} \times \dots $.
Therefore, we can choose a transversal $\left\{\sigma\right\}$ of the set of double cosets  $\Sigma_{i,j} \sigma \Sigma_{j,\ell} $ in $\Sigma_j$ so that the restriction of $\sigma$ to $\Sigma_{j(1)}$ is the identity. Now that we have such a transversal, we construct from it a transversal $\bar \sigma$ of the set of double cosets $\Sigma_{\bar\imath,\bar\jmath} \bar \sigma \Sigma_{\bar\jmath,\bar \ell} $ in $\Sigma_{ \bar\jmath}$ in the following way: $\sigma \mapsto \bar \sigma$, where $\bar \sigma_ {\mid_{\bar\jmath(t)}}=  \sigma_{ \mid_{ j(t)} }$, for $t \neq1$, and
$\bar \sigma_ {\mid_{\bar\jmath(1)}}= id$. This can be done because $j(t)= \bar\jmath(t)$, for $t\neq 1$, and $\Sigma_{\bar\imath,\bar\jmath}= \Sigma_{\bar\jmath(1)} \times \dots $. Now we have 
\begin{equation*} 
\phi\left ( \xi_{i,j}\xi_{j,\ell}\right) = \sum_{\sigma}\left[ \Sigma_{i \sigma,
\ell}: \Sigma_{i \sigma,j, \ell}\right]\xi_{\overline{\imath\sigma}, \bar \ell}\,, \quad
\mbox{and} \quad \phi\left ( \xi_{i,j}\right) \phi\left(\xi_{j,\ell}\right) =
\sum_{\sigma}\left[ \Sigma_{\bar\imath\, \bar \sigma,\bar \ell}:
\Sigma_{\bar\imath\, \bar\sigma,\bar\jmath, \bar \ell}\right]\xi_{\bar\imath\,\bar\sigma, \bar \ell}.
\end{equation*}
But  $\bar\imath\bar \sigma=(1, i_{\sigma(1)}, \dots ,i_{\sigma(s)}) =
\overline{\imath\sigma}$. Also, if we write $\Sigma_{i\sigma , \ell}=\Sigma_{\ell(1)} \times X $ and $\Sigma_{i\sigma ,j, \ell}=\Sigma_{\ell(1)} \times Y $, then $\Sigma_{\bar\imath\bar \sigma ,\bar \ell}=\Sigma_{\bar \ell(1)} \times X $ and $\Sigma_{\bar\imath\bar \sigma ,\bar\jmath, \bar \ell}=\Sigma_{\bar \ell(1)} \times Y $. Therefore $\left[ \Sigma_{i \sigma, \ell}: \Sigma_{i \sigma,j, \ell}\right]= \left[ \Sigma_{\bar\imath \bar \sigma,\bar \ell}: \Sigma_{\bar\imath \bar\sigma,\bar\jmath, \bar \ell}\right]$ for all $\sigma$ in the transversal.
\qedhere
\end{proof}

\section{Borel-Schur algebras of wild  representation type}
In this section we will show that all the
Borel-Schur algebras of infinite type, excluding the algebras
$S^+(2,5)$ over fields of characteristic $3$ and the algebras
$S^+(3,2)$ over fields of arbitrary characteristic, are indeed wild.
\subsection{The algebra $S^+(3,3)$}
In this subsection we prove that the algebra $S^+(3,3)$ has  wild representation
type. By Propositions~\ref{subalgebra}~and~\ref{idempotent},
this implies that the algebras $S^+(n,r)$ are wild for all $n\ge 3$ and
$r\ge 3$.

Consider the  subset $X$ of $\Lambda(3,3)$ marked by the black dots
below
\begin{equation*}
\xymatrix@C-2em@R-2em{
\stackrel{030}{\bullet} &&
\stackrel{120}{\bullet} &&
\stackrel{210}{\circ} &&
\stackrel{300}{\circ} \\
&
\stackrel{021}{\bullet} &&
\stackrel{111}{\bullet} &&
\stackrel{201}{\bullet} &&
\\ &&
\stackrel{012}{\circ} &&
\stackrel{102}{\circ} &&
\\ &&&
\stackrel{003}{\circ} &&
}
\end{equation*}
 Let $\quiver$ denote the quiver  of $S^+(3,3)$  and write $\quiver'$ for the subquiver  of $\quiver$ spanned by $X$.  Since existence of an arrow $\alpha \colon
\lambda \to \mu$  in $\quiver$  implies that $\mu$ dominates $\lambda$, we get that the subquiver
$\quiver'$  is convex. Hence by
Proposition~\ref{convex}, $\quiver'$ is the quiver of the basic algebra $eS^+(3,3)e$ for $e = \sum_{\lambda \in
X} \xi_\lambda$.

If $\ch\K=2$, $\quiver'$  has the form
\begin{equation}
\label{quiver332}
\xygraph{
		!{<0ex,0ex>;<11.7ex,0ex>:<0ex,13.5ex>::}
!{(0,0)} :^(0){030}^(1){120}^-{\xi_{122,222}}[rr]
:@{<-}^-{\xi_{122,123}}^(0.9){111}[rd]
:^-{\xi_{113,123}}^(1){201}[rr]:@{}|(0.2){,}[r]
[lll]:@{<-}_-{\xi_{123,223}}^(1){021}[ll]
:^-{\xi_{222,223}}[lu]
[rd]:@/_3ex/_(0.3){\xi_{113,223}}[rrrr]
}	
\end{equation}
where we labelled arrows by the corresponding basis elements  of
$eS^+(3,3)e$. Notice that~\eqref{quiver332} does not contain any path of
length greater than or equal to $3$ .
By direct computation we have that $\xi_{122,123} \xi_{123,223}
\not=\xi_{122,222} \xi_{222,223}$ and $\xi_{113,123} \xi_{123,223} =0$ in
$S^+(3,3)$, and so also in $eS^+(3,3)e$.
Therefore the category  $eS^+(3,3)e\md$ is equivalent to the category
$(\quiver', \xi_{113,123}\xi_{123,223})\rep$.
Denote by $\quiver''$ the quiver obtained from $\quiver'$ by removing the arrow
$\xi_{113,123}$. Since the path $\xi_{113,123}\xi_{123,223}$ does not belong to
$\quiver''$, we get  that
$\quiver''\rep$ embeds into $(\quiver',\xi_{113,123}\xi_{123,223})\rep$. Now
$\quiver''$ is a quiver without relations and oriented cycles. Moreover it is
neither a Dynkin nor an extended Dynkin diagram. Therefore $\quiver''$ is wild. Hence also $(\quiver',\xi_{113,123}\xi_{123,223})$ and $eS^+(3,3)e$ are wild.
This shows that $S^+(3,3)$ is wild when $\ch\K=2$.

When $\ch\K\ge 3$ the quiver $\quiver'$ defined above  has the form
\begin{equation}
\label{quiver333}
\xygraph{
		!{<0ex,0ex>;<11.7ex,0ex>:<0ex,13.5ex>::}
!{(0,0)} :^(0){030}^(1){120}^-{\xi_{122,222}}[rr]
:@{<-}^-{\xi_{122,123}}^(0.9){111}[rd]
:^-{\xi_{113,123}}^(1){201}[rr]:@{}|(0.2){.}[r]
[lll]:@{<-}_-{\xi_{123,223}}^(1){021}[ll]
:^-{\xi_{222,223}}[lu]
}	
\end{equation}
We  have again $\xi_{122,123} \xi_{123,223}
\not=\xi_{122,222} \xi_{222,223}$ in $S^+(3,3)$. Therefore $eS^+(3,3)e$ is
hereditary in this case. As~\eqref{quiver333} is not of Dynkin type we conclude
that $eS^+(3,3)e$ and, hence, also $S^+(3,3)$ have wild representation type.

\subsection{The algebra $S^+(4,2)$}
We consider now the algebra  $S^+(4,2)$ and show that it has wild representation type if
the characteristic of the base field $\K$ is $p\ge 3$. 
 In view of Proposition~\ref{subalgebra} and
Proposition~\ref{idempotent}, this will imply that the algebras $S^+(n,r)$ are  wild for all $n\ge 4$, $r\ge 2$, and $\ch \K \ge 3$.

We consider the idempotent
\begin{equation*}
e = \xi_{24,24} + \xi_{23,23} + \xi_{12,12}
\end{equation*}
in $S^+(4,2)$, and write $A = e S^+(4,2) e$. We are going to compute the quiver of $A$. It has
three vertices corresponding to the primitive idempotents $\xi_{12,12}$, $\xi_{23,23}$,
$\xi_{24,24}$. Now, to obtain a basis of rad $A$, note that
\begin{equation*}
\begin{aligned}
\xi_{23,23} A \xi_{24,24}& = \K \xi_{23,24}, \quad
\xi_{12,12} A \xi_{24,24}& = \K \xi_{12,24}\oplus \K \xi_{12,42}, \\
\xi_{12,12} A \xi_{23,23}& = \K \xi_{12,23} \oplus \K\xi_{12,32}.
\end{aligned}
\end{equation*} So $\left\{\,\xi_{23,24}, \xi_{12,24}, \xi_{12,42}, \xi_{12,23}, \xi_{12,32}\,\right\}$ is a basis of rad $A$.
On the other hand,
\begin{equation*}
\xi_{12,23} \xi_{23,24} = \xi_{12,24},\quad \xi_{12,32} \xi_{23,24} =
\xi_{12,42},
\end{equation*} which implies that rad$^2 A$ has basis $\,\left\{ \xi_{12,24}, \xi_{12,42}\, \right\}$.
Thus $A$ is the path algebra of the quiver
\begin{equation}
\label{subquiver42}
\xygraph{
		!{<0ex,0ex>;<7.8ex,0ex>:<0ex,9ex>::}
[rr]:[l]:@/^2ex/[l][r]:@/_2ex/[l]
	}
\end{equation}
without relations. Since \eqref{subquiver42} is neither a Dynkin nor an extended
Dynkin diagram, it has wild representation type. Hence
the algebra
$S^+(4,2)$ is wild.

\subsection{The algebras $S^+(2,r)$.}
Let $p$ be the characteristic of the base field. 
In this section we show that $S^+(2,p+1)$ has  wild representation type if
$p\ge 5$, and that the same result holds for $S^+(2,4)$  if $p=2$, and for $S^+(2,6)$  if $p=3$.
In Section~\ref{lastpar}  we will prove that $S^+(2,5)$ has tame representation type if $p=3$.  By Propositions~\ref{subalgebra}~and~\ref{idempotent}, this completes the classification of all
algebras $S^+(2,r)$ of infinite representation type.

We change slightly the notation introduced in Section~\ref{schur_preliminaries}, namely
we write $\alpha$ for $\alpha_0$, $\beta$ for $\alpha_1$, and
$\gamma$ for $\alpha_2$, for
the types of arrows in the quiver of  $S^+(2,r)$.
This is convenient since the types $\alpha_s$ with $s\ge 3$ will not appear in this
section.
Further, we will identify compositions $\lambda = (\lambda_1, \lambda_2) \in
\Lambda(2,r)$ with $\lambda_2$. This will not create an ambiguity, since
$r$ will be fixed in each particular case and $\lambda_1 = r - \lambda_2$.

\subsubsection{The algebra $S^+(2,p+1)$ over a field of characteristic  $p\ge 7$}
We described the quiver $\quiver$ of $S^+(2,p+1)$ and the corresponding
relations $R$ in Section~\ref{schur_preliminaries}.
Denote by $\quiver'$ the subquiver
\begin{equation}
	\label{quiver}
	\xygraph{
		!{<0ex,0ex>;<11ex,0ex>:<0ex,9ex>::}
[rrrrrrr]:^(0){4}|(0){\bullet}_{\alpha}[l]
:^(0){3}^(1){2}_{\alpha}[l]
:^(1){1}_{\alpha}[l]="r"
:^(1.2){0}_{\alpha}[l]="b"
[ur]="t":_(0){p+1}_(1){p}^{\alpha}[l]="l"
:_(1){p-1}^{\alpha}[l]
:_(1.0){p-2}^{\alpha}[l]
"t":^{\beta}"r" "l":_{\beta}"b"
	}	
\end{equation}
of $\quiver$.
We used here the fact that $\ch\K\ge 7$, as otherwise different vertices in~\eqref{quiver}
would have the same labels in $\Lambda(2,p+1)$, and so~\eqref{quiver}
would not be a subquiver of $\quiver$.
Now, $\ch\K\ge 7$ implies that the sets of relations $\rst{(\alpha)^p_\lambda}{\quiver'}$ and
$\rst{(\beta)^p_\lambda}{\quiver'}$ are empty for all
$\lambda \in \quiver'_0$.
Hence $\rst{R}{\quiver'}=\left\{ (\beta\alpha)_{p+1} - (\alpha \beta)_{p+1} \right\}$.
 The pair $(\quiver',\rst{R}{\quiver'})$ is isomorphic to the XXVIIIth
quiver in
Ringel's list in~\cite{Ri} of minimal wild quivers with one relation.
By Proposition~\ref{subquiver}, we get that $S^+(2,p+1)$ is wild over a field
of characteristic no less than $7$.

\subsubsection{The algebra $S^+(2,6)$ over a field of characteristic  $5$}
The quiver $\quiver$ of the algebra $S^+(2,6)$ over a field of characteristic
$5$  is given by
\begin{equation}
	\label{quiver265}
		\xygraph{
!{<0cm,0cm>;<12ex,0ex>:<0cm,8ex>::}
[rrrrr]:^(0.8){5}_-{\alpha}_(-0){6}[l]:^-{\alpha}_(1){4}[l]:^-{\alpha}_(1){3}[l]:^-{\alpha}_(0.9){2}[l]:_-{\alpha}_(1.1){1}[l]:^-{\alpha}^(1){0}[l]
[rrrrrr]:@/^6ex/^{\beta}[lllll]
[rrrr]:@/_6ex/_{\beta}[lllll]
}
\end{equation}
and the corresponding relations are
$R = \left\{\ (\alpha^5)_{6}\, ,\ (\alpha^5)_{5}\, ,\ (\alpha\beta-
\beta\alpha)_{6}\ \right\}$. 
Next we consider the quiver $\widetilde{\quiver}$
\begin{equation}
	\label{cover265}
	\xygraph{
!{<0cm,0cm>;<1.0cm,0cm>:<0cm,1.0cm>::}
[drrrrrrrrrrr]
:_(0){3''}_(1){2''}^{\alpha}[l]
:_(0.8){1''}^{\alpha}[l]
:^(1.0){0''}^{\alpha}[dl]
:@{<-}^{\beta}[ul]
:_(0){5''}_(1){4''}^{\alpha}[l]
:_(1){3'}^{\alpha}[l]
:_(1){2'}^{\alpha}[l]
:_(0.8){1'}^{\alpha}[l]
:^(1.0){0'}^{\alpha}[ld]
:@{<-}^{\beta}[ul]
:_(0){5'}_(1){4'}^{\alpha}[l]
:@{-}[d]:@{-}@/_0.3cm/[dr]:@{-}^{\alpha}
[rrrrrrrrr]:@{-}@/_0.3cm/[ur]:[u]
[ll]:@{<-}_(1.3){6''}_{\beta}[ul]:_{\alpha}[dl]
[llll]:@{<-}_(1.3){6'}_{\beta}[ul]:_{\alpha}[dl]
}
\end{equation}
with a free action of the symmetric group $\Sigma_2$ given by
interchanging $'$ with $''$.
The quiver \eqref{quiver265} is the quotient of \eqref{cover265} under this
action when we identify the orbit $\{s',\ s''\}$ with the vertex $s$
of the quiver~\eqref{quiver265}.
As the quotient map preserves the types of arrows, we get that
\begin{equation*}
R_{\Sigma_2} = \left\{\
(\alpha^5)_{6'}\, ,
(\alpha^5)_{6''}\, ,
(\alpha^5)_{5'}\, ,
(\alpha^5)_{5''}\, ,
(\alpha\beta - \beta\alpha)_{6'}\, ,\
(\alpha\beta - \beta\alpha)_{6''} \
 \right\}.
\end{equation*}
Denote by $\widetilde{\quiver}'$ the subquiver
\begin{equation}
	\label{subcover265}
	\xygraph{
!{<0cm,0cm>;<1.0cm,0cm>:<0cm,1.0cm>::}
[drrrrrrrrrrr]
:_(0){4'}^{\alpha}[l]
:_(0){3''}_(1){2''}^{\alpha}[l]
:_(0.8){1''}^{\alpha}[l]
:^(1.0){0''}^{\alpha}[dl]
:@{<-}^{\beta}[ul]
:_(0){5''}_(1){4''}^{\alpha}[l]
:_(1){3'}^{\alpha}[l]
[rrrr]:@{<-}_(1.3){6''}_{\beta}[ul]:_{\alpha}[dl]
}
\end{equation}
of $\widetilde{\quiver}$.
As $\widetilde{\quiver}'$ does not contain paths of type $\alpha^5$ and
$6'\not\in \widetilde{\quiver}_0'$, we get
that
\begin{equation*}
\rst{R_{\Sigma_2}}{\widetilde{\quiver'}} = \left\{ \ (\alpha\beta-\beta\alpha)_{6''} \right\}.
\end{equation*}
The pair $(\widetilde{\quiver}', \rst{R_{\Sigma_2}}{\widetilde{\quiver'}} )$ is isomorphic to the XVIIIth
quiver in Ringel's list in~\cite{Ri} of minimal wild quivers with one  relation.
By Theorem~\ref{covering_criterion},  the quiver $(\quiver,R)$ is wild.
This shows that $S^+(2,6)$ has wild representation type over  fields of
characteristic~$5$.
\subsubsection{The algebra $S^+(2,6)$ over a field of characteristic $3$.}
In this section we show that 
if the base field has characteristic $3$ then 
the algebra $S^+(2,6)$ is wild. The quiver $\quiver$ of $S^+(2,6)$ is
\begin{equation}
	 \label{quiver263}
	 \xygraph{
!{<0cm,0cm>;<1cm,0cm>:<0cm,1cm>::}
[rrrrrr]:^{\alpha}_(0)6[l]:^{\alpha}_(0)5_(1)4[l]:^{\alpha}_(1)3[l]:^{\alpha}_(1)2[l]:^{\alpha}_(1)1[l]:^{\alpha}_(1)0[l]
[rrrrr]:@/_5ex/_{\beta}[lll]
[rr]:@/^5ex/^{\beta}[lll]
[rr]:@/_5ex/_{\beta}[lll]
[rrrrrr]:@/^5ex/^{\beta}[lll]
	 }
 \end{equation}
and the corresponding relations are
\begin{equation*}
\begin{aligned}
R : = & \left\{
 (\alpha^3)_{6}\,,\
(\alpha^3)_{5}\,,\
(\alpha^3)_{4}\,,\
(\alpha^3)_{3}\, , \right.
\\[2ex]  &\phantom{=} \left.
(\alpha \beta -\beta\alpha)_{6}\,,\
(\alpha \beta -\beta\alpha)_{5}\,,\
(\alpha \beta -\beta\alpha)_{4}\
 \right\}.
\end{aligned}
\end{equation*}
 Let us consider the quiver $\widetilde{\quiver}$
 \begin{equation}
	 \label{cover263}
	 \xygraph{
	 !{<0cm,0cm>;<1.2cm,0cm>:<0cm,1.2cm>::}
[ddrrrrrr]
:^{\alpha}_(0){5''}_(0.8){4''}[l]
:^{\alpha}_(0.6){3''}[l]
:_{\alpha}^(1){2'}[l]
:_{\alpha}^(1){1'}[l]
:_{\alpha}^(1){0'}[l]
:^{\beta}@{<-}[u]
:@{-}@/_0.3cm/[ld]
:@{-}[d]
:^{\alpha}@{-}@/_0.3cm/[dr]
:@{-}[rrrrr]
:@(r,r)[u]
:_{\alpha}^(0){2''}^(1){1''}[l]
:_{\alpha}^(1){0''}[l]
:^{\beta}@{<-}[u]
[rr]
:^{\beta}[d] [lu]
:^{\beta}[d]
[ull]
:^{\beta}@{<-}[u]
:_{\alpha}_(0){5'}_(1){4'}[l]
:^{\alpha}_(0.7){3'}[l] [r]
:_{\beta}[d]
[rrrr]
:^{\alpha}@{<-}_(1){6''}[r]
:@{-}[u]
:^{\beta}@{-}@/_0.3cm/[ul]
:@{-}[llll]
:@/_0.3cm/[ld]
[rr]
:^{\alpha}@{<-}^(1){6'}[r]
:_{\beta}[d]
	 }
 \end{equation}
 with the action of $\Sigma_2$ given
  by swapping $'$ and $''$.
The quotient $\widetilde{\quiver}/\Sigma_2$ is isomorphic to the
quiver~\eqref{quiver263} and the canonical projection preserves the types of
arrows.
Therefore
\begin{equation*}
R_{\Sigma_2} = \left\{\, (\alpha^3)_{s'}\,,\ (\alpha^3)_{s''}\,,\ (\alpha\beta -\beta\alpha)_{t'}\,,\
(\alpha\beta-\beta\alpha)_{t''}\,\middle|\, 3\le s \le 6, \,4\le t \le 6
\right\}.
\end{equation*}
We consider the following subquiver $\widetilde{\quiver}'$ of \eqref{cover263}
\begin{equation}
\label{subquiver263}
 \xygraph{
	 !{<0cm,0cm>;<1.2cm,0cm>:<0cm,1.2cm>::}
[rrrrrrdd]
:_{\beta}@{<-}_(0.2){0'}[u]
:_{\alpha}_(0){3'}^(1){2''}[l]
:_{\alpha}^(1){1''}[l]
:_{\beta}@{<-}[u]
:_{\alpha}_(-0.2){4''}^(1){3''}[l]
:_{\alpha}^(1){2'}[l]
:^{\beta}@{<-}[u]
:_{\alpha}_(0){5'}_(1){4'}[l][r]
:^{\alpha}@{<-}^(1){6'}[r]
:^{\beta}[d]
}
\end{equation}
As $\widetilde{\quiver}'$ does not contain any path of type $\alpha^3$ and contains only one
square with the border paths of types $\alpha\beta$ and $\beta\alpha$, we get that $\rst{R}{\widetilde{\quiver}'} = \left\{\ (\alpha\beta-\beta\alpha)_{6'}\  \right\}$. 
The quiver with relations $(\widetilde{\quiver}',  \rst{R}{\widetilde{\quiver}'})$ is
isomorphic
to the XXIXth quiver in Ringel's list of minimal wild quivers with one
relation provided in~\cite{Ri}.
From Theorem~\ref{covering_criterion}, we conclude that $(\quiver, R)$ is wild.
This implies that $S^+(2,6)$ is wild over a field of characteristic $3$.

\subsubsection{The algebra $S^+(2,4)$ over a field of characteristic $2$.}

The last algebra to be analyzed in this section is $S^+(2,4)$ over a base field
of characteristic $2$. We will show that this  algebra is wild.

Given a quiver $\quiver$, the associated \emph{separated quiver} $\sep(\quiver)$
is the quiver with the set of vertices $\left\{\, v',\ v'' \,\middle|\, v\in
\quiver_0
\right\}$ and arrows $\bar\varepsilon \colon v' \to w''$ for every $\varepsilon\colon
v \to w$ in $\quiver$. Let $R$ be any set of relations for $\quiver$. 
According to Gabriel (see~\cite{gabriel_Dynkin}) the category of representations of
$\sep(\quiver)$ is equivalent to the category of representations of
$(\quiver,R)$ whose Loewy length does not exceed $2$. 
From this and from the classification of tame hereditary quivers it follows that if
$(\quiver,R)$ is tame then $\sep(\quiver)$ is a union of Dynkin and extended
Dynkin diagrams. 

The quiver of the algebra $S^+(2,4)$ over a base field of characteristic
$2$ is
\begin{equation*}
	\xygraph{
	!{<0cm,0cm>;<1.2cm,0cm>:<0cm,1.2cm>::}
	[rrrr]:_{\alpha}_(-0.2){4.}_(1)3[l]:^{\alpha}^(1)2[l]:^{\alpha}_(1)1[l]:_{\alpha}_(1.2)0[l]
	:@{<-}@/^4ex/[rr]^{\beta}:@{<-}@/^4ex/[rr]^{\beta}:@/^8ex/[llll]^{\gamma}
	[r]:@{<-}@/_4ex/[rr]_{\beta}
	} 
\end{equation*}
The corresponding separated quiver has three connected components: two isolated
vertices $0'$, $4''$, and 
\begin{equation*}
	\xygraph{
	!{<0cm,0cm>;<0.86cm,-0.5cm>:<0cm,1cm>::}
:^(0){1'}^(0.9){0''}[r]
:^(1){2'}@{<-}[ru]
:^(1){1''}[r]
:^(1){3'.}@{<-}[d]
:^(1){2''}[ld]
:^(0.9){4'}@{<-}[l]
:[u][d]:_(1){3''}[ld]
	} 
\end{equation*}
Since the above quiver is neither Dynkin nor extended Dynkin diagram, we
conclude that $S^+(2,4)$ is wild if the characteristic of the base field
is~$2$. 

\section{Tame representation type: the degeneration technique and $S^+(3,2)$.}

We proved in~\cite{finite} that the Borel-Schur algebra $S^+(3,2)$ has infinite
representation type
(independently of
the characteristic of the ground field $\K$).
The aim of this section is to show that $S^+(3,2)$ is tame. For this we will use degeneration
techniques developed by Gabriel in~\cite{gabrieldegen} and by Geiss in~\cite{geiss}.

Given a vector space $V$, denote by $\alg(V)$ the variety of associative algebra structures with identity
on $V$. Each such structure is determined by the multiplication map $\mu \colon V\otimes V \to V$. Hence
we can consider $\alg(V)$ as a  subset of the affine space $\Hom  _{\K}(V\otimes V, V)$.
The group $\Gl(V)$ acts on $\Hom  _{\K}(V\otimes V, V)$ by
$\,
g  \mu= g\circ \mu \circ (g^{-1} \otimes g^{-1}),
$ all $g \in \Gl(V)$. This action preserves $\alg(V)$.

A product $\mu_0 \in \alg(V)$ is called a \emph{degeneration} of
$\mu \in \alg(V)$ if $\mu_0$ lies in the Zariski closure of the
$\Gl(V)$-orbit of~$\mu$.

We will use the following result proved in~\cite{geiss}.
\begin{theorem}[{\cite{geiss}}]\label{geiss}
Let $\mu_0$ be a degeneration of $\mu$ in $\alg(V)$.
If $\mu_0$ is not wild, then the same holds for $\mu$.
\end{theorem}
One way to construct a degeneration of $\mu \in \alg(V)$ is as
follows. Fix a basis $\left\{\,v_1, \dots, v_m\, \right\}$ of $V$, such that
$v_1$ is the identity element for $\mu$. Then $\mu$ determines the
multiplication
constants $\gamma_{hl}^k$ by
\begin{equation*}
\mu (v_h \otimes v_l) = \sum_{k=1}^m \gamma_{hl}^k v_k.
\end{equation*}
Suppose we have a function $\phi\colon \left\{ 1,\dots,m \right\} \to \N$
such that $\phi(k) -\phi(h) - \phi(l) \ge 0$, for every triple $(h,l,k)$ satisfying $\gamma_{hl}^k \not=0$. We also assume that $\phi(1)=0$.   For each $t\in \K^*$, define the linear isomorphism $g_t \colon V \to V$ by
$\,
g_t (v_h) = t^{\phi(h)} v_h, \,\, h=1, \dots, m.
$
Then we obtain another algebra with  the product $\mu_t := g_t\mu$ satisfying
\begin{equation*}
(\mu_t)(v_h \otimes v_l ) = g_t \left( \mu ( g_t^{-1} v_h  \otimes g_t^{-1} v_l) \right) =
\sum_{k=1}^m t^{\phi(k)- \phi(h) - \phi(l) } \gamma_{hl}^k v_k.
\end{equation*}
Since $\phi(k) -\phi(h)-\phi(l)\ge 0$ if $\gamma^{k}_{hl}\not=0$, we can substitute $t$ by $0$ in this formula.  We obtain a new  product
$\mu_0$ which  is well defined. It follows from the associativity of the products
$\mu_t$ for $t\not=0$ that $\mu_0$ is associative. Since
$\phi(1)=0$, it is easy to check that $v_1$ is the identity for $\mu_0$.
Let $G$ be the subgroup $\left\{ g_t \middle| t \in \K^* \right\}$  of $\Gl(V)$. Then $ \mu_0$ is in the Zariski closure of $G\mu$. Henceforth it is also in the
Zariski closure of $\Gl(V)\mu$.

Now we apply this technique to the algebra $S^+(3,2)$.
Its quiver depends on the characteristic of the base field:
\begin{equation}
\label{quiver32}
\begin{array}{cc}
(a)\,\ch\K=2 & (b)\,\ch\K\not=2\\[1ex]
\xygraph{
		!{<0ex,0ex>;<6ex,0ex>:<0ex,7ex>::}
!{(1,2)} :|(-0.1){002}^(1){011}_(0.4){\xx{23}{33}}[lu] :_(0.3){\xx{22}{23}}[lu]
:^(-0.2){020}^(1){110}^{\xx{12}{22}}[rr]:^(1.1){200}^{\xx{11}{12}}[rr]
[llld]:_(1.1){101}^{\xx{13}{23}}[rr]:_(0.4){\xx{12}{13}}[lu]
[dd]:@/^6ex/^{\xx{22}{33}}[lluu]
:@/^5ex/^{\xx{11}{22}}[rrrr]
	} & 
\xygraph{
		!{<0ex,0ex>;<6ex,0ex>:<0ex,7ex>::}
!{(1,2)} :|(-0.1){002}^(1){011}^(0.4){\xx{23}{33}}[lu] :^{\xx{22}{23}}[lu]
:^(-0.2){020}^(1){110}^{\xx{12}{22}}[rr]:^(1){200}^{\xx{11}{12}}[rr]
[llld]:_(1.1){101}^{\xx{13}{23}}[rr]:_(0.4){\xx{12}{13}}[lu]
	} 
\end{array}
\end{equation}
Next we compute the multiplication table for $S^+(3,2)$ with respect to the
$\xi$-basis 
\bigskip

\renewcommand{\arraystretch}{1.2}
\begin{tabular}[h]{cccccccc}
\multicolumn{3}{c}{Multiplication at $(1,0,1)$} && \multicolumn{2}{c}{Multiplication at $(0,1,1)$} \\
& \multicolumn{1}{|c}{\xx{13}{23}} & \xx{13}{33} &&&
\multicolumn{1}{|c}{\xx{23}{33}} \\
\cline{1-3} \cline{5-6}
\multicolumn{1}{c|}{\xx{11}{13}} &
\xx{11}{23} & 2\xx{11}{33} & &
\multicolumn{1}{c|}{\xx{11}{23}} &
2 \xx{11}{33}
\\[1ex]
\multicolumn{1}{c|}{\xx{12}{13}} &
\xx{12}{23} & \xx{12}{33} &&
\multicolumn{1}{c|}{\xx{12}{32}} &
\cellcolor[gray]{0.9} \xx{12}{33} \\
&&&&
\multicolumn{1}{c|}{\xx{12}{23}} &
\xx{12}{33} \\
\multicolumn{3}{c}{Multiplication at $(0,2,0)$}
&&
\multicolumn{1}{c|}{\xx{13}{23}} &
\xx{13}{33} \\
& \multicolumn{1}{|c}{\xx{22}{33}}
& \xx{22}{23}
&&
\multicolumn{1}{c|}{\xx{22}{23}} &
\cellcolor[gray]{0.9}2 \xx{22}{33} \\
\cline{1-3}
\multicolumn{1}{c|}{\xx{11}{22}} &
\xx{11}{33} & \cellcolor[gray]{0.9} \xx{11}{23} \\
\multicolumn{1}{c|}{\xx{12}{22}} &
\cellcolor[gray]{0.9}\xx{12}{33} & \cellcolor[gray]{0.9} \xx{12}{23} + \xx{12}{32}\\ \\
\multicolumn{6}{c}{Multiplication at $(1,1,0)$} \\
 & \multicolumn{1}{|c}{\xx{12}{23}} &
\xx{12}{33} & \xx{12}{13} & \xx{12}{22} & \xx{12}{32} \\
\cline{1-6}
\multicolumn{1}{c|}{\xx{11}{12}} &
\xx{11}{23} & 2\xx{11}{33} & \phantom{=} \xx{11}{13} \phantom{=}&
\cellcolor[gray]{0.9} 2\xx{11}{22} &
\cellcolor[gray]{0.9} \xx{11}{23}
\end{tabular}
\label{t:first}
\renewcommand{\arraystretch}{1}
\bigskip

\noindent Note that we do not list
the trivial products $\xi_{i,j} \xi_{i',j'} = 0$ in case $j$ and $i'$ are not in the same $\Sigma_r$-orbit, and
$\xi_{i,i} \xi_{i,j} = \xi_{i,j}= \xi_{i,j} \xi_{j,j}$.
Given $t\in \K^*$, we define the diagonal automorphism  $g_t$ of $S^+(3,2)$
by multiplying
\begin{equation*}
\xx{11}{22},\ \xx{22}{33},\ \xx{11}{13},\ \xx{12}{13},\ \xx{13}{23},\
\xx{13}{33},
\end{equation*}
 with $t$,
\begin{equation*}
\xx{11}{33},\ \xx{11}{23},\ \xx{12}{23},\ \xx{12}{33}
\end{equation*}
with $t^{2}$,
and by fixing all the other basis elements. The induced product
$*_t$ coincides with the original product for all pairs of elements, except
the following ones:
\begin{equation*}
\begin{aligned}
\xx{12}{32} *_t \xx{23}{33} & = t^2 \xx{12}{33}, &
\xx{22}{23} *_t \xx{23}{33} & = 2 t \xx{22}{33}, \\[2ex]
\xx{11}{22} *_t \xx{22}{23} & =  t \xx{11}{23}, &
\xx{12}{22} *_t \xx{22}{33} & = t \xx{12}{33} \\[2ex]
\xx{11}{12} *_t \xx{12}{22} & = 2t \xx{11}{22}, &
\xx{11}{12} *_t \xx{12}{32} & = t^2 \xx{11}{23}
\end{aligned}
\end{equation*}
and
\begin{equation}
\label{long}
\xx{12}{22} *_t \xx{22}{23}  = \xx{12}{32} + t^2 \xx{12}{23}.
\end{equation}
We shaded the corresponding cells in the multiplication table for $S^+(3,2)$.
Now, by setting $t=0$, we get a new product $ *_0$ on the vector space
$S^+(3,2)$. Denote the resulting algebra by $B$.
We can identify from the multiplication table for $S^+(3,2)$ a basis of $\rad^2 B$.
Namely, every non-shaded cell in which the basis element has coefficient
$1$ give an element of $\rad^2 B$ independently of the characteristic of $\K$, i.e.
\begin{equation}
\label{list}
\xx{11}{23},\ \xx{12}{23},\ \xx{12}{33},\ \xx{13}{33},\
\xx{11}{33},\ \xx{11}{13}
\end{equation}
are in $\rad^2 B$. Now, the non-shaded cells in which the basis element has
coefficient $2$ could give  extra elements  of $\rad^2 B$ if
$\ch \K \not=2$. But   only  $\xx{11}{33}$ appears in these cells and it
is already in the list~\eqref{list}. One more element of $\rad^2 B$ comes
from~\eqref{long}, namely $\xx{12}{32}$. We obtain in this way a basis of $\rad^2 B$.
Taking the complement of the computed basis of $\rad^2 B$ inside the basis of  $\rad B$, we get
that the images of
\begin{equation*}
\xx{11}{12},\ \xx{11}{22},\ \xx{12}{22},\ \xx{12}{13},\ \xx{22}{23},\ \xx{22}{33},\
\xx{13}{23},\ \xx{23}{33}
\end{equation*}
by the canonical epimorphism $\rad B \to \rad B/\rad^2 B$ form a basis of $\rad
B/\rad^2 B$. Now, it is easy to verify, that the quiver of $B$ coincides with the
quiver~\eqref{quiver32}(a) of $S^+(3,2)$ over a field of characteristic $2$. 
From the multiplication table we know that the following products vanish
\begin{equation*}
\xx{22}{23}*_0\xx{23}{33},\quad \xx{12}{22}*_0\xx{22}{33},\quad
\xx{11}{22}*_0\xx{22}{23},\quad
\xx{11}{12}*_0\xx{12}{22}.
\end{equation*}
Therefore, by inspecting the quiver of $B$, we see that $B$ is
 a special biserial algebra.
By a classification result of Wald and Waschb\"{u}sch~\cite{wald}, we get  that
the representation type of $B$ is
  either finite
or tame. In other words $B$ is not wild. Hence by Proposition~\ref{geiss},  the
Borel-Schur algebra $S^+(3,2)$ does not have wild representation type.  We proved in~\cite{finite} that $S^+(3,2)$ has infinite representation
type, hence  $S^+(3,2)$ is tame.

\section{Tame representation type: the algebra $S^+(2,5)$ over a field of characteristic $3$.}\label{lastpar}

To finish classifying the representation type of  Borel-Schur algebras,  we need  to study  $S^+(2,5)$ over a field
$\K$ of characteristic $3$. In this section we show that this algebra has tame representation type.  We will use a combination of Auslander-Reiten
theory and poset representation theory  as described by  Ringel in~\cite{Ri}.

Given a finite dimensional algebra $A$ and an $A$-module $M$, the
\emph{one-point extension} algebra $A[M]$ is the matrix algebra
$\sm{A & M \\ 0 & \K}$.
This is relevant since any Borel-Schur algebra
$S^+(2, r)$ is a one-point extension algebra
	$A[M]$, where $A$ is isomorphic to  $S^+(2, r-1)$.
To see this, let $S:= S^+(2,r)$. Take the idempotent
$e = \xi_{(0,r)}$ of $S$, and let $\be = 1-e$.
From the relations for $S$ (see also~\cite[Lemma~6.2]{finite}), it is immediate that $\be S \be \cong
S^+(2,r-1)$.
As a left $S$-module $S$ can be decomposed into the direct sum
$S\be \oplus Se$. Note that, for any $(0,r) \neq \lambda \in \Lambda (2,r)$,  $\lambda $ dominates $(0,r)$. Therefore $eS\bar{e}=0$
and hence $S\be = \be S\be$, and
$Se = M\oplus \K e$, where $M={\rm rad} Se =\be Se$.
With this,  the product of two elements in $S$ has precisely
the form of the product in $A[M]$ with $A= \be S \be$.

\bigskip

The representation theory of a one-point extension algebra is often closely
related to  representations of a certain poset.
Given a poset $(\poset, \le)$, a $\poset$-space $(V,V_p)$ is a vector space $V$
together with subspaces $V_p$, $p \in \poset$, such that $p \le q$ implies $V_p \subset V_q$.
A  homomorphism between $\poset$-spaces
$(V,V_p)$ and $(W,W_p)$
is a linear map
$f\colon V \to W$ such that
 $f(V_p) \subset W_p$ for all $p\in \poset$.

Given a  natural number $k$, we also denote by $k$ the ordinal with $k$ elements.
Given posets $\poset_1$, \dots, $\poset_s$, we write $(\poset_1,\dots,\poset_s)$
for their disjoint union. We denote by $N$ the poset consisting of four
elements $t_1$, $t_2$ , $b_1$, $b_2$ with relations $t_1 < b_i$, $t_i < b_2$
for $i=1$, $2$. This poset can be visualized as follows
\begin{equation*}
\xymatrix{
t_1 \ar[d] \ar[dr] & t_2  \ar[d] \\ b_1 & b_2. 
}
\end{equation*}
Like for algebras, it is possible to define the representation type of a poset
$\poset$, with  the category of $\poset$-spaces taking the place of
the category of modules. 
It was proved by Nazarova in~\cite{nazarova} that every poset has either
finite, or tame, or wild representation type, and that these possibilities are
mutually exclusive. Moreover, she characterized the wild posets.

\begin{theorem}[{Nazarova,~\cite{nazarova}}]\label{nazarova}
The poset $\poset$ is of wild representation type if and only if $\poset$
contains as a full subposet one of the sets $(1,1,1,1,1)$, $(1,1,1,2)$,
$(2,2,3)$, $(1,3,4)$, $(1,2,6)$, or $(N,5)$.
\end{theorem}
We call the six  posets listed in this theorem \emph{ Nazarova posets}.

For a  finite dimensional algebra $A$,  denote by $\Gamma_A$ the Auslander-Reiten quiver of $A$.
Now consider a finite dimensional $A$-module $M$, and the functor
$h_M:= {\rm Hom}_A(M,-)$.
Define $\Gamma_M$ to be the subquiver of $\Gamma_A$ whose vertices are
given by the indecomposable $A$-modules $N$ with $h_M(N) \not=0$, and arrows are given by
the irreducible morphisms $f$ in $\Gamma_A$ such that $h_M(f)\not=0$.

\begin{proposition}\label{nazarova2}
Let $A$ be a finite dimensional algebra of finite representation type and
$M$ an indecomposable $A$-module.
If $\dim h_M(N) \le 1$ for all indecomposable $A$-modules $N$,  then
$\Gamma_M$ is the Hasse diagram of a poset $\poset_M$. In this case,
 the representation type of
	$A[M]$ coincides with the representation type of
$\poset_M$.
\end{proposition}

For a proof we refer to Section~2 of~\cite{Ri}, see also
Section~2 of ~\cite{ringel1099},  which discusses this in more detail.
The underlying set of  $\poset_M$ consists of  the isomorphism classes $[U]$
of indecomposable $A$-modules $U$ such that $h_M(U)\neq 0$. The partial order
is defined by
$[U]\leq [V]$
provided there is an irreducible
map $f: U\to V$ such that $h_M(f)$ is non-zero.
With our assumptions, for $\Gamma_A$ without multiple arrows,
$h_M(f)\neq 0$ if and only if
$h_M$ induces an injective map $h_M(U)\to h_M(V)$.
To minimize the number of symbols we write $U$ instead of $[U]$ for an element
of the poset.

We apply these results to
$S=S^+(2, 5)$ for $\K$ of characteristic 3.
We have seen that this is the one-point extension $A[M]$ with $A=S^+(2, 4)$
and where $M$ is the radical of $Se$, as described above.
We have proved in \cite{finite}
that $A=S^+(2, 4)$ has finite representation type,
and computed its Auslander-Reiten  quiver.
We reproduce it in Figure~\ref{ARQ}.
{\begin{figure}
	\includegraphics[trim=4cm 4cm 6cm 4cm]{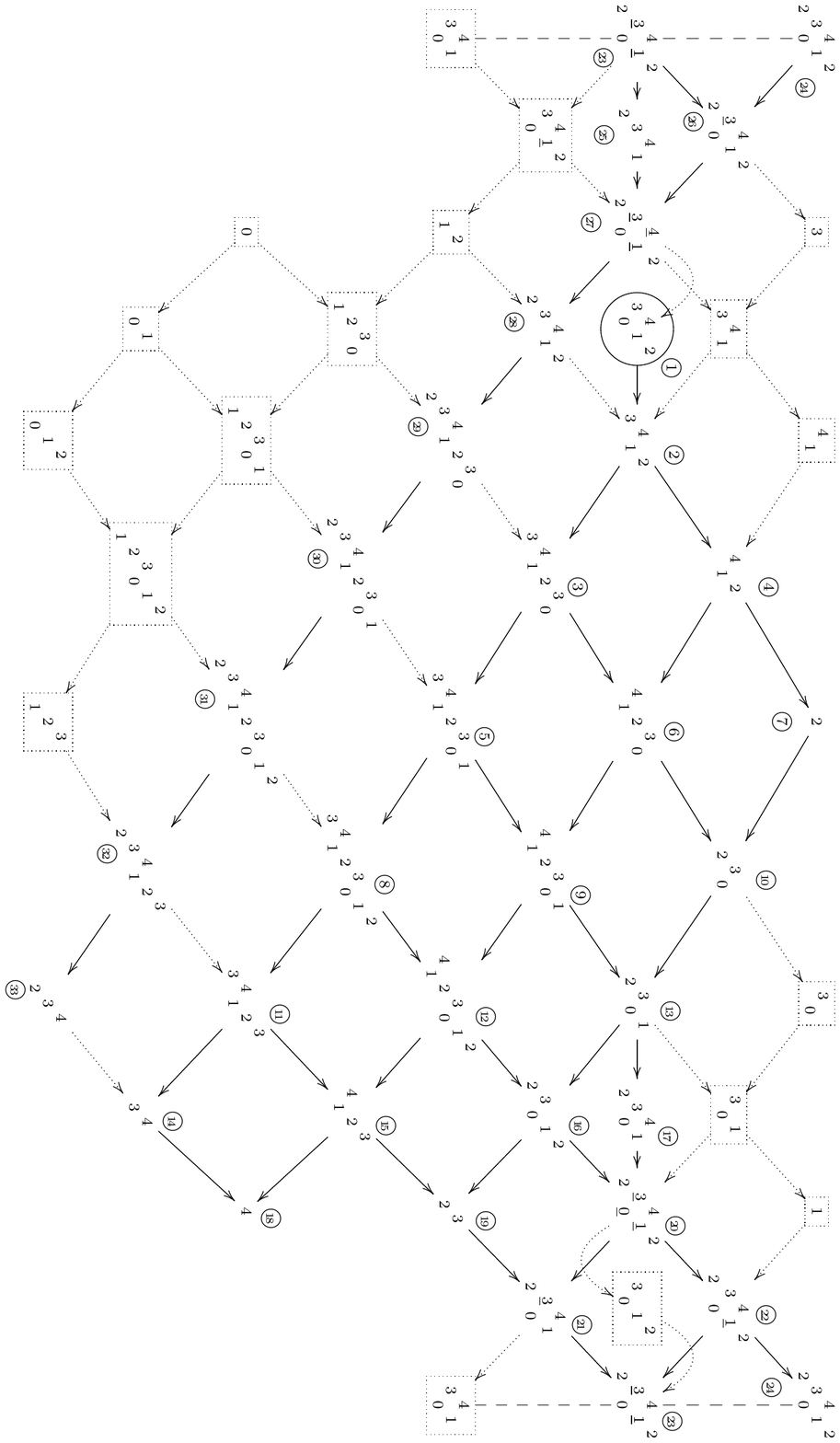}\caption{AR-quiver for $S^+(2,4)$,~$p=3$.\quad\quad\quad\quad\quad\quad\quad\quad\quad\quad\quad\quad\quad\quad\quad\quad\quad\quad\quad\quad\quad\quad\quad\quad\quad\quad\quad}
	\label{ARQ}
\end{figure}
}
It is $90^\circ$ clockwise rotated for typographical reasons.
 The module $M$ is framed with a circle.
We also labelled the modules in $\Gamma_M$ with encircled numbers for future
reference, and we will write $M_t$ for the module in $\Gamma_M$ labelled with $t$. 
In  Subsection~\ref{CompGamma} we will  prove the following:
\begin{itemize}
\item {\it The modules not in
$\Gamma_M$ are framed with dots.}
\item {\it  The arrows in $\Gamma_M$ are solid and all other arrows are dotted.}
\end{itemize}
The right side of the quiver is glued to the left side along the dashed line.
The resulting poset  is redrawn in Figure~\ref{poset}.
\begin{figure}
\begin{gather}
\xymatrix@R1.5ex@C1.5ex{
1 \ar[d] \\ 
2 \ar[d] \ar[rd] \\
3 \ar[d] \ar[rd] & 4 \ar[d] \ar[rd]  \\
5 \ar[d] \ar[rd] & 6 \ar[d] \ar[rd] & 7 \ar[d] \\
8 \ar[d] \ar[rd] & 9 \ar[d] \ar[rd] & 10 \ar[d]  \\
11 \ar[d] \ar[rd] & 12 \ar[d] \ar[rd] & 13 \ar[d] \ar[rd] \\
14  \ar[rd] & 15 \ar[d] \ar[rd] & 16 \ar[d] \ar[rd] & 17 \ar[d] \\
& 18 & 19 \ar[rd] & 20 \ar[d] \ar[rd] \\
&&& 21 \ar[rd] & 22 \ar[d] \ar[rd] \\
&&&& 23 \ar[d] \ar[rd] & 24 \ar[d] \\
&&&& 25 \ar[rd] & 26 \ar[d] \\
&&&&& 27 \ar@{.}[dd] \\\\
&&&&& 33
}
\end{gather}
\caption{Poset $\poset_M$}
\label{poset}
\end{figure}
Later we will show that this poset has tame representation type.

Now we explain the notation we use for modules in Figure~\ref{ARQ} as much as  we need.
Once more we will  write $\alpha$ for $\alpha_0$, $\beta$ for $\alpha_1$, and $\lambda_2$ for $(\lambda_1, \lambda_2) \in \Lambda(2,4)$.
As proved in \cite{finite},
the quiver $Q$ of $A$ is of  the form
\begin{equation*}
 	\xygraph{
	!{<0cm,0cm>;<1.2cm,0cm>:<0cm,1.2cm>::}
	[rrrr]:_{\alpha}_(-0.2){4.}_(1)3[l]:^{\alpha}^(1)2[l]:_{\alpha}_(1)1[l]:^{\alpha}_(1.2)0[l]
	:@{<-}@/^4ex/[rrr]^{\beta}[ll]:@{<-}@/_4ex/[rrr]_{\beta}}
\end{equation*}
and  $A$ is defined by the relations
$\alpha^3=0$, $( \beta\alpha)_4 = (\alpha\beta)_4$.  
Consider the factor algebra
$\bar{A}:= A/ ( \beta\alpha)_4$.
Note that  the $\bar{A}$-modules
are precisely those $A$-modules on which $( \beta\alpha)_4$ (hence $(\alpha\beta)_4$)
acts as zero.
The algebra $\bar{A}$
 is a special biserial algebra, of finite type. Thus indecomposable 
$\bar{A}$-modules (and the corresponding $A$-modules) are string modules and can be described
by  quivers as in~\cite{ringel_dihedral} or in~\cite[Chapter~II]{lnm1428}. 
Consider for example $\s{ &&& 3 \\ 4 & &2 && 0 \\ & 1}$.
This stands for the 5-dimensional module with basis  $\left\{v_1, \dots, v_5\right\}$. The
 canonical idempotents of $A$ act as $\xi_{(4-t,t)} v_t=v_t$, $t=1, \dots, 4$,  the arrows act as follows
\begin{equation*}
\begin{aligned}
\beta_4 v_4= v_1, \quad \alpha_2 v_2=v_1, \quad \alpha_3 v_3=v_2,\quad \beta_3 v_3=v_0,
\end{aligned}
\end{equation*}
and anything else acts as zero.
This module has minimal generators $v_4, v_3$, and its largest semisimple
submodule is spanned by $v_1$ and $v_0$.

 The above string modules  account for  most vertices
in Figure 1.
We do not need explicit descriptions of the other modules, 
except the module $\s{& 4 \\ 3 && 1 \\ &0}$. 
It has a basis $\left\{\, v_i \,\middle|\, 0\le i\le 4 \right\}$, the canonical
idempotents of $A$ act again as $\xi_{(4-t,t)} v_t = v_t$, the arrows act as
follows
\begin{equation*}
\begin{aligned}
\beta_4 v_4 = v_1,\quad \alpha_4 v_4 = v_3,\quad \alpha_1 v_1 = v_0,\quad
\beta_3 v_3 = v_0,
\end{aligned}
\end{equation*}
and anything else acts as zero. 

\subsection{Computing $\Gamma_{M}$}\label{CompGamma}

To compute $\Gamma_{M}$ we use the following reduction that simplifies calculations.
Let $L_0$ denote the (simple) socle of $M$ and write $\bar{M} = M/L_0$. Notice
that $\bar{M} \cong \mdl{2}$. We claim
that the canonical projection  $\pi \colon M \twoheadrightarrow \bar{M}$ induces an isomorphism 
${\Hom_A (\bar{M}, X) \cong \Hom_A(M,X)}$, 
for every indecomposable $A$-module $X$ which is not isomorphic to $M$. 
Since $M \to \bar{M}$ is an epimorphism, the map 
\begin{equation*}
\begin{aligned}
\Hom_A (\bar{M}, X) & \to \Hom_A(M,X)\\
f &\mapsto f \circ \pi
\end{aligned}
\end{equation*}
is an inclusion, and its image can be identified with those $\theta \colon M
\to X$ such that $\theta(L_0) =0$. 
Suppose this map is not surjective. Then there is 
  $\theta\colon M\to X$ such that $\theta(L_0)\not= 0$. As $L_0$ is the socle of
$M$, the map $\theta$
is injective, and it splits, since
$M$ is injective as an $A$-module. As we assumed that $X$ is indecomposable, we must have
$X\cong M$, which contradicts our assumption that $X \not\cong M$. 

 On the part of $\Gamma_A$ with $\bar{A}$-modules one can compute
$h:=\Hom_A(\bar{M}, -)$ 
using the string module presentation (recall that $\bar{A}=A/ ( \beta\alpha)_4$).
We deal now with the part involving non-string modules. 
For convenience, we
draw the relevant part of $\Gamma_A$ (recall that the labelling refers to
Figure~\ref{ARQ}). It is also convenient to include several 
string modules, in particular, $\mdl{13}$, $\mdl{16}$, $\mdl{19}$, and
$\mdl{28}$.  
\begin{equation*}
\begin{gathered}
\xymatrix{
   &&
      *+[F.]{1} \ar@{..>}[rd]
   &&
      \mdl{24} \ar[rd]
   &&
      *+[F.]{3} \ar@{..>}[rd]
\\  
   &
      *+[F.]{\s{3&&1 \\ &  0}}   \ar@{..>}[ru]\ar@{..>}[rd] 
   && 
      \mdl{22}\ar[ru]\ar[rd] 
   &&
      \mdl{26}\ar@{..>}[ru]\ar[rd] 
   &&
      *+[F.]{\s{& 4 \\3&&1}} 
\\
      \mdl{13}\ar@{..>}[ru]\ar[rd] \ar[r]
   & 
      \mdl{17} \ar[r] 
   & 
      \mdl{20} \ar[ru]\ar[rd] \ar@{..>}[r]
   & 
      *+[F.]{\s{&&&2 \\ 3 && 1 \\ &  0}}  \ar@{..>}[r] 
   & 
      \mdl{23}\ar[ru]\ar@{..>}[rd] \ar[r]
   & 
      \mdl{25}  \ar[r]
   & 
      \mdl{27} \ar@{..>}[ru]\ar[rd] \ar@{..>}[r]
   & 
      M 
\\
   & 
      \mdl{16} \ar[ru]\ar[rd]
   && 
      \mdl{21}\ar[ru]\ar@{..>}[rd]
   && 
      *+[F.]{\s{& 4&& 2 \\ 3 && \td{1} \\ & 0 }} \ar@{..>}[ru] 
      \ar@{..>}[rd]
   &&
      \mdl{28}
\\
   &&
      \mdl{19} \ar[ru]
   &&
      *+[F.]{\s{&4 \\ 3 && 1 \\ & 0}} \ar@{..>}[ru]
    &&
      *+[F.]{\s{&2 \\  1}}  \ar@{..>}[ru]
}
\end{gathered}
\end{equation*}
\begin{enumerate}
\item We exploit Auslander-Reiten sequences which have terms $U$,  where
$U$ is a string module and $h(U)=0$.
In each case, isomorphisms are given by $h(f)$ for a relevant irreducible
map, using the fact  that $h$ is a functor.

We start with the  path from ${\module{20}}$ via ${\module{24}}$ to ${\module{28}}$.
Applying $h$ to the Auslander-Reiten sequence ending in ${\module{22}}$ gives
$h({\module{20}})\cong h({\module{22}})$, and applying it to the Auslander-Reiten sequence ending in ${\module{24}}$ shows that
$h({\module{22}}) \cong h({\module{24}})$.
The same type of argument shows  that
$$h({\module{24}})\cong h({\module{26}}) \cong h({\module{27}}).
$$
We claim that $h({\module{27}}) \cong h({\module{28}})$. 
Consider the Auslander-Reiten sequence $0\to U\to Z\to V\to 0$ ending in the
module $V = \s{ &2\\ 1}$.
 Note that $Z\cong \tau({\module{28}})$.
Then $h(U)=0$, as the (simple) socle of $U$ is not a composition factor of $\bar{M}$. As well, $h(V)=0$, and hence $h(Z)=0$. This implies that $h({\module{27}}) \cong h({\module{28}})$, by applying  $h$ to the Auslander-Reiten sequence ending in ${\module{28}}$.
Using the fact that {\module{28}}  is a string module, we check that $h({\module{28}})\cong \mathbb{K}$.
\item Considering the Auslander-Reiten sequence starting with $\mdl{19}= \s{& 3
\\ 2}$  gives $\mathbb{K}\cong h(\mdl{19})\cong h({\module{21}})$.
The Auslander-Reiten sequence starting with ${\module{21}}$ gives $h({\module{21}}) \cong h({\module{23}})$,
which therefore also is 1-dimensional.
\item The Auslander-Reiten sequence ending with ${\module{25}}$  has middle
term ${\module{23}}$. Since $h$ takes the end term to zero we get
$h({\module{23}})\cong h({\module{25}})$.
Similarly the Auslander-Reiten sequence starting with $\mdl{17}$ has indecomposable
middle term ${\module{20}}$, and $h$ maps the end term to zero. Hence
$h(\mdl{17})\cong
h({\module{20}}) \cong \K$.

\item Considering the Auslander-Reiten sequence ending in $\bar{M}$, we
conclude that $h(M)\cong h(\bar{M}) \cong \K$. 
\end{enumerate}

In total, for each non-string module $W$ which is not framed,  we have that
$h(W)$ is one-dimensional
and most of the irreducible maps are taken to isomorphisms.

The remaining irreducible maps can be dealt with, either they are
injective between modules which are taken to $\mathbb{K}$ by $h$, or one can
use the mesh relations to deduce that $h$ takes them to isomorphisms.

\subsection{$\Gamma_M$ is tame}

From now, we only work with the poset $\Gamma_M$, and we use  the labels for
elements as shown in Figure~\ref{poset}.
Note that, since $S^+(2, 5)$ has infinite representation type
by~Theorem~\ref{nossa}, we know by Proposition 
~\ref{nazarova2} that the poset $\Gamma_M$ is not of finite type. 

Now we prove that $\Gamma_M$ is not wild. By  Theorem~\ref{nazarova}, we must show that
it does not contain a Nazarova subposet.
We can see directly that $\Gamma_M$ does not contain
five incomparable points, that is, $(1, 1, 1, 1, 1)$
does not occur.
To exclude the other five Nazarova subposets we use some reductions.

We describe first the strategy.
More generally, let $\cP$ be any poset.  
We wish to show that some disjoint unions $(Y, Z)$, with $Z$ of width
no less than $2$, are not  full subposets of $\cP$.

For any subposet $W$ of $\cP$, we write
\begin{equation*}
\begin{aligned}
C_W:= \{ s\in \cP \mid s \mbox{\ is not comparable with any } w\in W\}.
\end{aligned}
\end{equation*}
If $(Y,Z)$ is a full subposet of $\cP$,
then $Z$ is contained in $C_Y$.
Denote by $U$  the convex hull of $Y$, then $C_Y= C_U$. So we get that $Z\subset C_U$.
Moreover, if $X$ is any subposet of $U$ then $C_U\subseteq C_X$ and therefore
 $C_X$ should contain the subposet $Z$. Thus if $Z$ is of width no less than
$2$, we get that the same property holds for $C_X$. This suggests to have a list of  \emph{test subsets}  $X$ with $C_X$ of width no
less than $2$.

So let $\cP=\Gamma_M$. We will determine all subposets $X$ of size 3 of the form $X= \{ x, y, z\}$ with $x<y<z$ where $x,y,z$ are neighbours with respect to $<$, and we will
use this list for the strategy as above.
We refer to these as minimal triples.

Take such $X$.
Then, referring to the diagram,  $X$ is either ``vertical",  or $X$ is ``diagonal", or 
$X$ is at the edge, by this we mean  one of the posets
$$\{ 10, 13, 17\}, \ \ \{ 17, 20, 22\}, \  \{ 21, 23, 25\},
$$
or $X$ is not convex and its convex hull is a set $V= X\cup \{ w\}$ with $x<w<z$. 
We list now all such  posets $X$ and the corresponding $C_X$ for which  $C_X$
has width no less than  $2$. 

(1) There are six such $X$ which are vertical:
\begin{equation*}
\begin{gathered}
\xymatrix@R2ex@C1.5em{
  5 \ar[d] && 6 \ar[d] && 7 \ar[d] && 8 \ar[d] && 9 \ar[d] && 10 \ar[d] \\ 
  8 \ar[d] && 9 \ar[d] && 10 \ar[d] && 11 \ar[d] && 12 \ar[d] && 13 \ar[d] \\ 
 11  && 12  && 13  && 14  && 15  && 16  
}
\end{gathered}\end{equation*}
The corresponding subsets $C_X$ are 
\begin{equation*}
\begin{gathered}
\xymatrix@R2ex@C1em{
4 \ar[rd] \ar[d] 
\\ 
6 \ar[rd] & 7   \ar[d] 
\\ 
& 10
}
\end{gathered},\quad
\begin{gathered}
\xymatrix@R1ex@C0.5em{
 & 11 \ar[dd] \\ 7 \\ & 14
}
\end{gathered},\quad
\begin{gathered}
\xymatrix@R1.6ex@C1em{
8 \ar[rd] \ar[d] \\
11 \ar[d] \ar[rd] & 12 \ar[d] \\
14 \ar[rd] & 15 \ar[d] \\
& 18 
}
\end{gathered}
,\quad
\begin{gathered}
\xymatrix@R1.6ex@C1em{
4 \ar[rd] \ar[d] \\
6 \ar[d] \ar[rd] & 7 \ar[d] \\
9 \ar[rd] & 10 \ar[d] \\
& 13  \ar[rd]\\ 
&& 17
}
\end{gathered},\quad
\begin{gathered}
\xymatrix@R1ex@C0.5em{
 & 7 \ar[dd] \\ 14 \\ & 10
}
\end{gathered}
,\quad
\begin{gathered}
\xymatrix@R2ex@C1em{
11 \ar[rd] \ar[d] 
\\ 
14 \ar[rd] & 15   \ar[d] 
\\ 
& 18 
}
\end{gathered}
\end{equation*}
(2) There are four such $X$ which are diagonal:
\begin{equation*}
\xymatrix@R2ex@C1em{
15 \ar[rd] &&& 16 \ar[rd]&&& 20 \ar[rd]&&& 19 \ar[rd]\\
& 19 \ar[rd]&&& 20 \ar[rd]&&& 22 \ar[rd]&&& 21 \ar[rd]\\ 
&& 21 &&& 22 &&& 24 &&& 23 
}
\end{equation*}
The corresponding subsets $C_X$ are 
\begin{equation*}
\begin{gathered}
\xymatrix@R1ex@C0.5em{
 & 22 \ar[dd] \\ 14 \\ & 24 
}
\end{gathered}
,\quad \quad\quad
\begin{gathered}
\xymatrix@R2ex@C1em{
11 \ar[rd] \ar[d] 
\\ 
14 \ar[rd] & 15   \ar[d] 
\\ 
& 18
}
\end{gathered}
,\quad\quad\quad
\begin{gathered}
\xymatrix@R2ex@C1em{
11 \ar[rd] \ar[d] 
\\ 
14 \ar[rd] & 15   \ar[d]  \ar[rd] 
\\ 
& 18 & 19 
}
\end{gathered}
,\quad\quad
\begin{gathered}
\xymatrix@R1ex@C0.5em{
 & 14 \ar[dd] \\ 24 \\ & 18 
}
\end{gathered}
\end{equation*}
(3) There are three posets $X$ at the edge of $\poset$ with $C_X$ of width $2$:
\begin{equation*}
\xymatrix@R2ex@C1em{
10 \ar[d] &&&& 17 \ar[d] &&&& 21 \ar[dr] \\
13 \ar[rd] &&&& 20 \ar[rd] &&&&& 23 \ar[d] \\
& 17 &&&& 22 &&&& 25
}
\end{equation*} 
The corresponding subsets $C_X$ are
\begin{equation*}
\begin{gathered}
\xymatrix@R2ex@C1em{
8 \ar[rd] \ar[d] \\
11 \ar[rd] \ar[d] & 12 \ar[d]
\\ 
14 \ar[rd] & 15   \ar[d] 
\\ 
& 18
}
\end{gathered}
,\quad\quad
\begin{gathered}
\xymatrix@R2ex@C1em{
11 \ar[rd] \ar[d] 
\\ 
14 \ar[rd] & 15   \ar[d]  \ar[rd] 
\\ 
& 18 & 19 
}
\end{gathered}
,\quad\quad
\begin{gathered}
\xymatrix@R1ex@C0.5em{
 & 14 \ar[dd] \\ 24 \\ & 18 
}
\end{gathered}
\end{equation*}
(4) There are  three sets $V_t= X_t\cup \{w_t\}$, for $t=1, 2, 3$,  for which
$C_{X_t} (=C_{V_t})$ has width no less than $2$:

\begin{equation*}
\begin{aligned}
V_1 = 
\begin{gathered}
\xymatrix@R2ex@C1em{
11 \ar[rd] \ar[d] 
\\ 
14 \ar[rd] & 15   \ar[d] 
\\ 
& 18
}
\end{gathered}
,\quad
 V_2 =
\begin{gathered}
\xymatrix@R2ex@C1em{
13 \ar[rd] \ar[d] 
\\ 
16 \ar[rd] & 17   \ar[d] 
\\ 
& 20 
}
\end{gathered}
,\quad
  V_3 =
\begin{gathered}
\xymatrix@R2ex@C1em{
12 \ar[rd] \ar[d] 
\\ 
15 \ar[rd] & 16   \ar[d] 
\\ 
& 19 
}
\end{gathered}
\end{aligned}
\end{equation*}
For these, the corresponding subsets $C_{V_t}$ are, respectively, 
\begin{equation*}
\begin{gathered}
\xymatrix@R2ex@C1em{
7 \ar[d]  \\
10 \ar[d] \\
13 \ar[rd]\ar[d] \\ 
16 \ar[rd] & 17 \ar[d] \\
& 20 \ar[rd] \\ && 22 \ar[rd] \\ &&& 24
}
\end{gathered}
,\quad 
\begin{gathered}
\xymatrix@R2ex@C1em{
11 \ar[rd] \ar[d] 
\\ 
14 \ar[rd] & 15   \ar[d] 
\\ 
& 18
}
\end{gathered}
,\quad \quad\quad
\begin{gathered}
\xymatrix@R2ex@C1em{
14  & 17 
}
\end{gathered}. 
\end{equation*}

\noindent With this preparation, we can exclude Nazarova posets.
\bigskip

\noindent (I) We claim that $\cP$ does not have a subposet isomorphic to $(1, 3, 4)$. Assume for a contradiction there is a subposet $(Y, Z)$
where $Y\cong (3)$ and $Z\cong (1, 4)$. 
Let $U$ be the convex hull of $Y$, so $C_U = C_Y$ and $Z\subseteq C_U$. 
Then $U$ contains a minimal triple $X$  and then $C_U\subseteq C_X$. Therefore   $C_X$ contains a subposet isomorphic to $(1, 4)$. 
The only sets $X$ in our list  such that $C_X$ contains a subposet $(4)$   are
$X=\{ 8, 11, 14\}$, $X=\{ 7, 10, 13\}$,  $X=\{ 10, 13, 17\}$,   
and $X\subset  V_1$, but then  $C_X$ does not contain a subposet
$(1, 4)$, a contradiction.

\bigskip

\noindent (II)  We claim that $\cP$ does not have a subposet isomorphic to $(2, 2, 3)$. Suppose there is a subposet $(Y, Z)$ with $Y\cong (3)$ and $Z \cong (2, 2)$.
Let $U$ be the convex hull of $Y$, so $C_U = C_Y$ and $Z\subseteq C_U$. 
Then $U$ contains a minimal triple $X$  and then $C_U\subseteq C_X$ and $Z\subseteq C_X$ so that $C_X$ contains a subposet isomorphic to $(2,2)$. 
Our list does not contain such a minimal triple, a contradiction.

\bigskip

\noindent (III) We claim that $\cP$ does not have a subposet  isomorphic to $(N, 5)$. 
Assume for a contradiction that there is a subposet $(Y, Z)$ with $Y$ isomorphic to $(5)$ and $Z$ isomorphic to $N$.
Let $U$ be the convex hull of $Y$, so that $C_U=C_Y$. Then  $Y$  (and $U$) contains at least three
different minimal triples $X$, and $C_U\subseteq C_X$. Each of these $C_X$ must contain
the same copy of $N$ as a subposet.
From our list, the only subposets isomorphic to $N$ which occur as subsets of more than one of such $C_X$ are 
\begin{equation*}
N_1:= 
\begin{gathered}
\xymatrix@R2ex@C1em{
14 \ar[rd] & 15 \ar[d] \ar[rd] \\
& 18 & 19
}
\quad\quad\quad
\end{gathered}
N_2:= 
\begin{gathered}
\xymatrix@R2ex@C1em{
11 \ar[d]\ar[rd] & 12 \ar[d] \\
 14 & 15.
}
\end{gathered}
\end{equation*}
We see $C_{N_1}= \{ 17, 20, 22, 24\}$  and 
$C_{N_2} = \{ 7, 10, 13, 17\}$, which are too small to contain the subposet $(5)$, and we have a contradiction.

\bigskip

\noindent (IV)  We claim that $\cP$ does not contain a subposet isomorphic to $(1, 2, 6)$. Suppose we have a subposet $(Y, Z)$ with $Y\cong  (6)$ and $Z\cong (1, 2)$.
Let $U$ be the convex hull of $Y$, so that $C_Y=C_U$.  Assume first that $U$ has
a subposet $V$ of size four which is the union of two minimal triples $x <y <z$
and $x < w < z$.
Then
$Z\subseteq C_U\subseteq C_V$. Now,  from part (4) of the list, $C_V$ has width $\geq 2$  for $V_1, V_2$ and $V_3$ but in each case
$C_V$ does not contain $(1, 2)$, a contradiction.
This shows that $U$ (and $Y$) does not contain such a $V$. This implies that $Y$ is either ``vertical" or ``diagonal" (using that $Y$ has size $6$). If $Y$ is vertical
then it can only be $\{ 4 < 6 < \ldots < 18\}$  but then $C_Y=\emptyset$. If $Y$ is diagonal then its smallest element is $8$ or $11$, and then $C_Y=\emptyset$ as well.
This shows that no subposet isomorphic to $(1, 2, 6)$ exists.

\bigskip

\noindent (V)  We claim that $\cP$ does not contain a subposet $Y$ isomorphic to $(1, 1, 1, 2)$.
If so then this contains the unique subposet of $\cP$ isomorphic to $(1, 1, 1, 1)$, which is $Z= \{ 14, 15, 16, 17\}$.
Then $Y$ is the union of $Z$ together with precisely one
element $s$ not in this set. In each case there
are two distinct elements of $Z$ comparable with $s$,
so that $Y$ is not the disjoint
union of $(2)$ with $(1, 1, 1)$, a contradiction.
$\Box$
\bigskip

We proved that $\Gamma_M$ does not contain any Nazarova subposet.  So, by  Theorem~\ref{nazarova}, we have that $\Gamma_M$ is not wild.
Since, by Proposition~\ref{nazarova2} and Theorem~\ref{nossa}, we know that it is not of finite type, we conclude that $\Gamma_M$ is tame.

\section*{Acknowledgement}
This work was partially
supported by the Centre for Mathematics of the University of Coimbra --
UID/MAT/00324/2019, funded by the Portuguese Government through FCT/MEC and
co-funded by the European Regional Development Fund through the Partnership
Agreement PT2020.
This research was also supported by the program 'Research in Pairs' by
the Mathematical Forschungsinstitut Oberwolfach in 2016.
The third author's position is financed by FCT via CEECIND/04092/2017. 

\bibliographystyle{abbrv}

\bibliography{oberwolfach}

\end{document}